\documentclass{imamat}

\gridframe{N}

\jno{dri017}

\usepackage{modnatbib}

\usepackage{latexsym,amssymb,lastpage}
\usepackage{graphicx,amsfonts}
\usepackage{times,mathptmx,bm,amsmath,amsthm}
\usepackage{dcolumn}

\renewenvironment{proof}[1][Proof]{\noindent\textit{#1. } }{\hfill$\square$}

 \newtheoremstyle{theorem}{6pt}{6pt}{\rm}{}{\sffamily}{ }{ }{}
 \theoremstyle{theorem}
\newtheorem{theorem}{\sc Theorem}[section]

  \newtheoremstyle{thm}{6pt}{6pt}{\rm}{}{\sffamily}{ }{ }{}
 \theoremstyle{thm}

 \newtheoremstyle{lemma}{6pt}{6pt}{\rm}{}{\sffamily}{ }{ }{}
 \theoremstyle{lemma}
 
 \newtheoremstyle{lem}{6pt}{6pt}{\rm}{}{\sffamily}{ }{ }{}
 \theoremstyle{lem}

\newtheoremstyle{case}{6pt}{6pt}{\rm}{}{}{. }{ }{}
 \theoremstyle{case}

 \newtheoremstyle{statement}{6pt}{6pt}{\rm}{}{\sffamily}{ }{ }{}
\theoremstyle{statement}

 \newtheoremstyle{corollary}{6pt}{6pt}{\rm}{}{\sffamily}{ }{ }{}
 \theoremstyle{corollary}
 \newtheorem{corollary}{\sc Corollary}[section]
 
  \newtheoremstyle{defi}{6pt}{6pt}{\rm}{}{\sffamily}{ }{ }{}
 \theoremstyle{defi}

  \newtheoremstyle{cor}{6pt}{6pt}{\rm}{}{\sffamily}{ }{ }{}
 \theoremstyle{cor}

\newtheoremstyle{example}{6pt}{6pt}{\rm}{}{\sffamily}{ }{ }{}
\theoremstyle{example}

\newtheorem{proposition}[theorem]{\sc Proposition}

\newtheoremstyle{remark}{6pt}{6pt}{\rm}{}{\sffamily}{ }{ }{}
\theoremstyle{remark}
\newtheorem{remark}{\sc Remark}[section]

\newtheoremstyle{approximation}{6pt}{6pt}{\rm}{}{\sffamily}{ }{ }{}
\theoremstyle{approximation}

\newtheoremstyle{scheme}{6pt}{6pt}{\rm}{}{\sffamily}{ }{ }{}
\theoremstyle{scheme}

\newtheoremstyle{Algorithm}{6pt}{6pt}{\rm}{}{\sffamily}{ }{ }{}
\theoremstyle{Algorithm}

 \newtheoremstyle{Remark}{6pt}{6pt}{\rm}{}{\sffamily}{ }{ }{}
 \theoremstyle{Remark}

\newtheoremstyle{Lemma}{6pt}{6pt}{\rm}{}{\sffamily}{ }{ }{}
\theoremstyle{Lemma}

\newtheoremstyle{Assumption}{6pt}{6pt}{\rm}{}{\sffamily}{ }{ }{}
\theoremstyle{Assumption}

\newtheoremstyle{Proposition}{6pt}{6pt}{\rm}{}{\sffamily}{ }{ }{}
\theoremstyle{Proposition}

\newtheoremstyle{prop}{6pt}{6pt}{\rm}{}{\sffamily}{ }{ }{}
\theoremstyle{prop}

\newtheoremstyle{rem}{6pt}{6pt}{\rm}{}{\sffamily}{ }{ }{}
 \theoremstyle{rem}

\newtheoremstyle{hypo}{6pt}{6pt}{\rm}{}{\sffamily}{ }{ }{}
 \theoremstyle{hypo}

  \newtheoremstyle{Step}{6pt}{6pt}{\rm}{}{}{ }{ }{}
 \theoremstyle{Step}

 \newtheoremstyle{lema}{6pt}{6pt}{\rm}{}{\sffamily}{ }{ }{}
 \theoremstyle{lema}

\renewcommand{\theequation}{\thesection.\arabic{equation}}
\numberwithin{equation}{section}

\newcommand{\R}{\mathbb{R}}

\DeclareMathOperator{\lip}{Lip}

\begin{document}
\title{A numerical procedure and unified formulation for the adjoint approach in hyperbolic PDE-constrained optimal control problems}

\author{
{\sc Gino I. Montecinos}$^*$\\[2pt]
Centro de Modelamiento Matem\'atico (CMM), Universidad de Chile, \\[6pt]
Beauchef 851, Torre Norte, Piso 7, Santiago, Chile \\[6pt] 
{\sc Juan C. L\'opez-R\'ios}\\[2pt]
Escuela de Ciencias Matem\'aticas y Tecnolog\'ia Inform\'atica, YACHAY TECH, \\[6pt]
San Miguel de Urcuqu\'i, Hacienda San Jos\'e s/n, Ecuador \\[6pt] 
{\sc Jaime H. Ortega}\\[2pt]
Centro de Modelamiento Matem\'atico (CMM) and Departamento de Ingenier\'ia
Matem\'atica, Universidad de Chile, \\[6pt]
Beauchef 851, Torre Norte, Piso 5, Santiago, Chile \\[6pt] 
{\sc Rodrigo Lecaros}\\[2pt]
Departamento de Matem\'atica, Universidad T\'ecnica Federico Santa Mar\'ia, \\[6pt]
Casilla 110-V, Valpara\'iso, Chile \\[6pt] 
}
\pagestyle{headings}
\markboth{Montecinos, L\'opez-R\'ios, Ortega and Lecaros}{\rm A unified formulation for the adjoint approach in hyperbolic PDE-constrained optimal control problems}

\maketitle


\begin{abstract}
{The present paper aims at providing a numerical strategy to deal with PDE-constrained optimization problems solved with the adjoint method. It is done through out a unified formulation of the constraint PDE and the adjoint model. The resulting model is a non-conservative hyperbolic system and thus a finite volume scheme is proposed to solve it. In this form, the scheme sets in a single frame both constraint PDE and adjoint model. The forward and backward evolutions are controlled by a single parameter $\eta$ and a stable time step is obtained only once at each optimization iteration. The methodology requires the complete eigenstructure of the system as well as the gradient of the cost functional. Numerical tests evidence the applicability of the present technique.}
{The adjoint method, PDE-constrained optimal control, hyperbolic conservation laws.}
\end{abstract}

\section{Introduction}

We are concerned with PDE-constrained optimal control problems. In general optimization problems are made up of the following ingredients; i) state variables, ii) design parameters, iii) objectives or cost functionals and iv) constraints that candidate state and design parameters are required to satisfy.  The optimization problem is then to find state and design parameters that minimize the objective functional subject to the requirement that constraints are satisfied, see \cite{Andersson:1998a} for further details. PDE-constrained optimization problems are those where the constraint consists of partial differential equations. These optimization problems arise naturally in control theory and inverse problems. Are usually employed as a methodology to obtain an approximate solution often from numerical simulations, see \cite{Knopoff:2013a, Kang:2005a} to mention but a few. Mathematical properties of these problems, as existence and ill-possedness, have been widely studied in the literature, see for instance \cite{Abergel1990,Hinze:2008a} and they will not be of interest in this work. In this paper, we are interested on generating a procedure to obtain numerical solutions to be as general as possible.  There are several successful strategies to solve these types of minimization problems: i) the one-shot or Lagrange multiplier method ii) based on sensitivity equations and iii) based on adjoint equations. Of interest in this work is the strategy based on adjoint equations. The interested reader can find further information regarding approaches i) and ii) in \cite{Andersson:1998a, Gunzburger:2003a} and references therein. 
One-shot methods cannot be used to solve most practical optimization problems especially involving time-dependent partial differential equations. Methods ii) and iii) are iterative processes involving the gradient of the objective functional, since the functional involves the solution of partial differential equations the task of obtaining the gradient becomes cumbersome, thus a safe option is just the use of approximations of them. However, in these problems even the approximation of gradients is a very challenging issue. Sensitivity equations require the differentiation of the state equation with respect to the design parameters. The simple form to achieve this is using a difference quotient approximation similarly to finite difference approximation as suggested in \cite{Andersson:1998a}. This strategy can be successfully applied for steady state models and for a finite and small number of design parameters. It becomes prohibitive for time-dependent cases. In this sense, the adjoint method is the suitable alternative. 

The adjoint method consists of the use of a linearized equation derived from the constraint PDE, it provides a new set of adjoint states, and the gradient of the cost functional can be completely determined in terms of the state variables, design parameters and adjoint states. The equation for the state variables, as well as for the adjoint parameter, is solved only once. It makes the adjoint method to be more efficient than the sensitivity method, see \cite{Hinze:2008a} for a detailed discussion about the performance of both approaches.  In this work, we are going to consider the adjoint approach for solving optimization problems related to hyperbolic type PDE's. Additionally, for hyperbolic equations, the design parameters, which are parameters of models, can be incorporated into the governing equation as a new variable, and then the problem of finding parameters can be cast into a problem of finding the initial condition.

The constraint PDE normally is written in a conservation  form  and thus suitable schemes are those corresponding to the family of conservative methods. On the other hand, the adjoint models are linearization of the constraint PDE, but this is a quasilinear type model for adjoint variables and thus they are also of hyperbolic type, but these models are non-conservative in a mathematical sense. It is a very important issue from a numerical point of view, the right description of the weave propagation has to be dealt in the frame of the path conservative methods for partial differential equations involving non-conservative products, see \cite{Pares:2006a, CastroM:2006a} for a detailed study of non-conservative partial differential equations with a particular emphasis on hyperbolic types.

As the constraint PDE and the adjoint model are both hyperbolic, they can be written as a unified model incorporating both the constraint PDE and the adjoint model in a single system of balance laws. In this paper, we back up on the unified formulation to set a numerical methodology able to simultaneously deal with the conservative form of constraint PDE and the adjoint model. The aim is to provide a general frame to deal with optimization problems through the adjoint method. Moreover, the present approach works even if the constraint PDE is non-conservative. In the literature, there exist several works dealing with minimization considering hyperbolic conservation laws, see  \cite{James:1999a, James:2008a, Castro:2008a-1, Lecaros:2014a,Ulbrich:2002a} to mention but a few.  However, to best of our knowledge, the present approach has not been reported previously and this incorporates in a natural way the non-conservative structure of the adjoint model.

On the other hand, regarding the solution of hyperbolic conservation laws, one of the successful methods are those based on the finite volume approach, which makes use of numerical fluxes at the interfaces of computational cells, providing a compact one-step evolutionary formula.  The accuracy of the approximation via this approach depends on the degree of resolution of numerical fluxes. One approach for achieving high resolution is by mean of the use of the so called Generalized Riemann Problem (GRP), see \cite{Castro:2008a,Montecinos:2012a} for further details and comparison among existing GRP solvers. These are building blocks of a class of high-order numerical methods known as ADER schemes \cite{Toro:2001c,Toro:2002a}. The numerical scheme presented in this work belong to the class of ADER methods. The scheme uses a modified version of the Osher-Solomon Riemann solver \cite{Osher:1982a} presented by Dumber and Toro, \cite{Dumbser:2010b}, which is able to account for the non-conservative terms, using the so-called path conservative approach, it is a second order method for hyperbolic problems in conventional finite volume setting and it inherits the properties of stability, well-balance, consistency and robustness. However, the implementation has to be adapted for accounting the correct wave propagation for the evolution of both, the PDE constraint and the adjoint model.  

In order to assess the performance of the present method, we carry out comparisons between the present scheme and a conventional method based on \cite{Lellouche:1994a} for solving PDE-constrained optimization problems. We adapt the scheme in \cite{Lellouche:1994a} implementing an upwind type scheme for the solution of the constraint and for the adjoint model, the same finite difference approach is kept but it is carried out in an explicit time evolution to be consistent with the present method which is globally explicit in time.

The rest of the paper is organized as follows. In Section \ref{PDE-constrint} the unified formulation of the primal dual strategy for a general system of conservation laws is presented, after the formal derivation of the adjoint system. We also state a general theorem providing the existence of classical solutions for the unified system and study a related optimal control problem. In Section \ref{section:numericalmethod}, a numerical scheme for solving conservative and non-conservative hyperbolic laws is presented and, as a consequence, a scheme for solving the unified primal-adjoint system. In Section \ref{sec:numerical-examples} some specific numerical examples are presented using the numerical scheme from Section \ref{section:numericalmethod}. Finally in Section \ref{sec:conclusions} some conclusions are drawn.

\section{PDE-constrained optimization problem}\label{PDE-constrint}

The purpose of this section is to provide a unified formulation of the primal dual strategy, raised from the study of certain optimal control problems concerning the finding of a parameter on a general conservation law.

First, we define the hyperbolic PDE to be considered and we prove that the finding of a parameter on this system is reduced to find an initial condition. Next we compute the adjoint system and set the unified formulation. Then we provide some conditions to establish the existence and uniqueness of classical solutions of the unified system.  Finally we set the control problem and prove the existence of an optimal solution.

%
We consider the system
\begin{eqnarray}
\label{eq:primal:0}
\left.
\begin{array}{c}
\partial_t \mathbf{U} + \partial_x \mathbf{R}(\mathbf{ U},\mathbf{b} ) = \mathbf{L} ( \mathbf{U,\mathbf{b}} ) + \tilde{\bf B}(\mathbf{U},\mathbf{b}) \partial_x \mathbf{b} \;, \;t\in [0,T]\;,\\
\mathbf{U}(x,0)=
\mathbf{H}(x) \;,
\end{array}
\right\}
\end{eqnarray}
where $T>0$ is a finite time, $\mathbf{H}( x)\in \mathbb{R}^{m}$ is a prescribed initial condition, $\mathbf{U}\in \mathbb{R}^{m}$ is a vector of states, $\mathbf{R}(\mathbf{ U},\mathbf{b}  )\in \mathbb{R}^{m}$ a flux function, $\mathbf{L} ( \mathbf{U},\mathbf{b} )\in \mathbb{R}^{m} $ is a source function, $\mathbf{b}(x)\in \mathbb{R}^{n}$ is a vector accounting for model parameters and the expression $\tilde{\bf B}(\mathbf{U},\mathbf{b}) \partial_x \mathbf{b}$ represents source terms which may also include the influence of parameter derivatives. Here, parameters $\mathbf{b}$ only depend on space. Since one main goal is to provide a strategy for finding $\mathbf{b}$ as the minimum of a given functional $J$, let us write $J$ as

\begin{equation}
\label{eq:functional:0}
J(\mathbf{U},\mathbf{b})=\frac{1}{2}\int_0^T\int_{\R}|\psi(\mathbf{U},\mathbf{b})-\bar{\psi}|^2dxdt+\frac{\alpha}{2}\|\mathbf{b}\|_{H^k}^2,
\end{equation}
where $\psi(\mathbf{U},\mathbf{b}) $ is some scalar field, representing some measurement operator, similarly $\bar{\psi}$ is some available measurement, $\alpha$ is a non-negative constant and $k$ is an exponent to be determined.

In an optimization setting we identify the state variables $\mathbf{U}$, design parameters $\mathbf{b}$, objectives or cost functional given by (\ref{eq:functional:0}) and  constraints consisting of (\ref{eq:primal:0}) that candidate state $\mathbf{U}$  and design parameters  $\mathbf{b}$ are required to satisfy.

Notice that, parameters $ \mathbf{b}$ only depend on space, then $ \partial_t \mathbf{b} (x)  = 0$. So, we can include $\mathbf{b}$ as a new variable of system  (\ref{eq:primal:0}), as follows

\begin{eqnarray}
\label{eq:primal:1}
\left.
\begin{array}{c}
\partial_t \mathbf{U} + \partial_x \mathbf{R}(\mathbf{ U},\mathbf{b} ) = \mathbf{L} ( \mathbf{U,\mathbf{b}} ) +  \tilde{\bf B} (\mathbf{U},\mathbf{b})\partial_x \mathbf{b} \;,\\
\\
\partial_t  \mathbf{b}= \mathbf{0}\;.
\end{array}
\right\}
\end{eqnarray}

So, written in the vector form,  (\ref{eq:primal:1}) becomes 
\begin{eqnarray}
\label{eq:primal:2}
\begin{array}{c}

\partial_t 
\tilde{\mathbf{U}} + 
\partial_x 

\tilde{\mathbf{R}} (\tilde{\mathbf{ U}}) 
+
\tilde{\bf M}  \partial_x  \tilde{\mathbf{U}} 
 = 

\tilde{ \mathbf{L} }( \tilde{\mathbf{U}} )
  
\;,
\end{array}
\end{eqnarray}
where
\begin{eqnarray}
\label{eq:primal:3}
\begin{array}{c}

\tilde{\mathbf{U}} = 
\left[
\begin{array}{c}
\mathbf{U} \\
\mathbf{b}
\end{array}
\right]
\;,
\;
\tilde{\mathbf{R}}(\tilde{\mathbf{U}})
=
\left[
\begin{array}{c}
\mathbf{R}(\mathbf{ U},\mathbf{b} ) \\
\mathbf{0}
\end{array}
\right]\;, 
\\

\tilde{\mathbf{L} }(\tilde{\mathbf{U}}) = 
\left[
\begin{array}{c}
 \mathbf{L} ( \mathbf{U},\mathbf{b} ) \\
 \mathbf{0}
\end{array}
\right] 
\;, \;
\tilde{  \mathbf{M} } (\tilde{\mathbf{U}}) = 
\left[
\begin{array}{cc}
\mathbf{0} & - \tilde{\bf B} (\mathbf{U},\mathbf{b})\\
\mathbf{0} & \mathbf{0} \\
\end{array}
\right] .

\end{array}
\end{eqnarray}

In this new system, $\mathbf{b}(x)$ is completely fixed once an initial condition is provided, so the task of finding parameters in a model system as  (\ref{eq:primal:0}) is equivalent to find the initial condition of the enlarged model system (\ref{eq:primal:1}),
henceforth referred to as {\it constraint PDE}. From now on, instead of system \eqref{eq:primal:0} we are going to consider system
\begin{eqnarray}
\label{eq:primal:01}
\left.
\begin{array}{c}
\partial_t 
\tilde{\mathbf{U}} + 
\partial_x 
\tilde{\mathbf{R}} (\tilde{\mathbf{ U}}) 
+
\tilde{\bf M}  \partial_x  \tilde{\mathbf{U}} 
 = 
\tilde{ \mathbf{L} }( \tilde{\mathbf{U}} ) 
\;, \;t\in [0,T]\;,\\
\tilde{\mathbf{U}}(x,0)=
\tilde{\mathbf{H}}(x) \;,
\end{array}
\right\}
\end{eqnarray}
where $\tilde{\mathbf{H}}(x)=[\mathbf{H},\mathbf{0}]^T$. The following result provides the conditions on system (\ref{eq:primal:0}) which turns system (\ref{eq:primal:01}) to be hyperbolic.

\begin{proposition}
\label{Prop1}
Let (\ref{eq:primal:0}) be a hyperbolic system in the variable $\mathbf{U}$, with $ \{\lambda_1, ...,\lambda_m\}$ the corresponding eigenvalues of $\partial \mathbf{R} / \partial \mathbf{U} $, where $\lambda_i \neq 0$. Then (\ref{eq:primal:01}) is also a hyperbolic system in the variables $ \tilde{\mathbf{R}}(\tilde{\mathbf{U}}) = [\mathbf{U},\mathbf{b}]^T$.
The set of eigenvalues is given by 
$$\{\lambda_1, ...,\lambda_m, \underbrace{0,...,0}_{n\;times} \}$$ 
and the corresponding eigenvectors are given by
$$\{ [\mathbf{v}_1, \mathbf{0}_{n}]^T, ...,[\mathbf{v}_m, \mathbf{0}_{n}]^T, [\tilde{\bf V}_1, \mathbf{e}_{1} ]^T,...,[\tilde{\bf V}_n, \mathbf{e}_{n} ]^T \} \;,$$ 
with
$\tilde{\bf V}_j = -  (\partial \mathbf{R} / \partial \mathbf{U} )^{-1}( \partial \mathbf{R} / \partial \mathbf{b} - \tilde{\bf B} )_j ,$
where $( \partial \mathbf{R} / \partial \mathbf{b} - \tilde{\bf B} )_j$ is the $j-th $ column vector of the matrix $( \partial \mathbf{R} / \partial \mathbf{b} - \tilde{\bf B} ).$ 
Here, $\mathbf{v}_i$ is an eigenvector associated with $\lambda_i$, $\mathbf{0}_m$ is the zero vector in $\mathbb{R}^{m}$, $\mathbf{e}_i$ is the $it$h canonical vector of $\mathbb{R}^{n}$. 
\end{proposition}
\begin{proof}
The Jacobian matrix of system (\ref{eq:primal:01}) is given by
\begin{eqnarray}
\begin{array}{c}
\mathbf{J_R} :=  
\partial \tilde{\mathbf{R}} / \partial \tilde{\mathbf{U}}+
\tilde{\bf M} 
=
\left[
\begin{array}{cc}
\partial \mathbf{R} / \partial \mathbf{U} & ( \partial \mathbf{R} / \partial \mathbf{b} - \tilde{\bf B} )\\
\mathbf{0} & \mathbf{0}
\end{array}
\right]\;.
\end{array}
\end{eqnarray}

Therefore
\begin{eqnarray}
\begin{array}{c}
det (  \mathbf{J_R} - \lambda \mathbf{I}_{n+m} )  = - \lambda^n det(\partial \mathbf{R} / \partial \mathbf{U} - \mathbf{I}_{m} )  = -\lambda^n p(\lambda) \;, 
\end{array}
\end{eqnarray}
where    $\mathbf{I}_{n+m} $ and $\mathbf{I}_{m} $ are the identity matrix in $\mathbb{R}^{n+m}$ and $\mathbb{R}^{m}$, respectively.  Here, $p(\lambda)$ is the characteristic polynomial of  $\partial \mathbf{R} / \partial \mathbf{U} - \mathbf{I}_{m} $ which has $m$ roots, because (\ref{eq:primal:0}) is hyperbolic. This proves
$$\{\lambda_1, ...,\lambda_m, \underbrace{0,...,0}_{n\;times} \}$$  are the eigenvalues of $ \mathbf{J_R}$.

To complete the proof, we must obtain a set of $m+n$ linearly independent eigenvectors. That means, for each $\mu \in  \{\lambda_1, ...,\lambda_m, \underbrace{0,...,0}_{n\;times} \} $ we look for $\mathbf{W}$ such that 
\begin{eqnarray}
\begin{array}{c}
\mathbf{J_R}  \mathbf{W}  = \mu   \mathbf{W}\;.
\end{array}
\end{eqnarray}

Here,  we obtain eigenvectors of the form $ \mathbf{W} = [\tilde{\bf V}, \tilde{\bf W}]^T$, so in a block-wise fashion we look for vectors 
$\tilde{\bf V}$ and  $\tilde{\bf W}$ such that
\begin{eqnarray}
\label{eq:eigenstruct:0}
\left.
\begin{array}{c}
(\partial \mathbf{R} / \partial \mathbf{U} )\tilde{\bf V} +( \partial \mathbf{R} / \partial \mathbf{b} - \tilde{\bf B} ) \tilde{\bf W}
= 
\mu \tilde{\bf V} \;, \\
\mathbf{0} = \mu \tilde{\bf W}  \;.

\end{array}
\right\}
\end{eqnarray}

Notice that, setting $\mu  = \lambda_i$ and $ \tilde{\bf V}  = {\bf v}_i $, then (\ref{eq:eigenstruct:0}) is satisfied for 
$  \tilde{\bf W} = \mathbf{0}_{n}$. It provides $m$ independent vectors $\{ [\mathbf{v}_1, \mathbf{0}_{n}]^T, ...,[\mathbf{v}_m, \mathbf{0}_{n}]^T \}$.  On the other hand, if $\mu = 0$ we obtain 
\begin{eqnarray}
\label{eq:eigenstruct:1}
\begin{array}{c}
(\partial \mathbf{R} / \partial \mathbf{U} )\tilde{\bf V} +( \partial \mathbf{R} / \partial \mathbf{b} - \tilde{\bf B} ) \tilde{\bf W}
=
\mathbf{0}  \;.

\end{array}
\end{eqnarray}

If we set $\tilde{\bf W} = \mathbf{e}_j$,  then
\begin{eqnarray}
\label{eq:eigenstruct:2}
\begin{array}{c}
(\partial \mathbf{R} / \partial \mathbf{U} )\tilde{\bf V} = - ( \partial \mathbf{R} / \partial \mathbf{b} - \tilde{\bf B} )_j
  \;,
\end{array}
\end{eqnarray}
where $( \partial \mathbf{R} / \partial \mathbf{b} - \tilde{\bf B} )_j$ is the $j-th $ column vector of the matrix $( \partial \mathbf{R} / \partial \mathbf{b} - \tilde{\bf B} ).$  Since $ \partial \mathbf{R} / \partial \mathbf{U} $ is invertible, there exists $\tilde{\bf V}_j\in \mathbb{R}^m$ such that 
$$ \tilde{\bf V}_j = - (\partial \mathbf{R} / \partial \mathbf{U} )^{-1}( \partial \mathbf{R} / \partial \mathbf{b} - \tilde{\bf B} )_j \;.$$

Then the proof holds.  
\end{proof}


\begin{remark}
Notice that the previous result ensure that extended system (\ref{eq:primal:01}) is hyperbolic, whenever system (\ref{eq:primal:0}) is. However, it requires (\ref{eq:primal:0}) to have an invertible Jacobian matrix. It is not guaranteed in the general case. However, the cases of interest in this study do so. Otherwise, the verification has to be done case by case.
\end{remark}

As a consequence, the problem of finding a parameter vector for system $(\ref{eq:primal:0})$ becomes an inverse problem where the aim is to find an initial condition for $\mathbf{b}$ which is a conventional variable in a system of hyperbolic balance laws.

\subsection{The adjoint method for solving PDE-constrained optimization problems}
\label{Section2.1}
%

%

In this section we derive a system of first order optimality conditions for system \eqref{eq:primal:01}, subject to the functional \eqref{eq:functional:0} with $\alpha = 0$. Even though these calculations are formal, we will provide the corresponding analysis in sections \ref{ExSol} and \ref{OptCont}. At this point our main interest is to present the unified primal-dual formulation and the main issues related with.
Since it is well known for this strategy \cite{Hinze:2008a}, the so-called adjoint state is, in some sense, a linearization of the constraint PDE. The dependence of the sought parameter $\mathbf{b}$ in \eqref{eq:primal:01} is implicit but, as was mentioned above, can be seen as the tracking of an initial condition.

As we did in Proposition \eqref{Prop1}, for $i=1:m+n$, let us write \eqref{eq:primal:01} as
\begin{eqnarray}
\label{eq:adjoint:-2}
\left.
\begin{array}{c}
\partial_t \tilde{\bf U}_i + \sum_{j = 1}^{m+n} \mathbf{J_R}_{i,j}\partial_x \tilde{\bf U}_j =  \tilde{\bf L}_i \;, \; t\in [0,T]\;,\\
\tilde{\mathbf{U}}_i(x,0)=
\tilde{\mathbf{H}}_i(x) \;.
\end{array}
\right\}
\end{eqnarray}

Then
\begin{eqnarray}
\label{eq:adjoint:-1}
\begin{array}{c}
\partial_t \delta\tilde{\bf U}_i
 + \sum_{j = 1}^{m+n} 
 \mathbf{J_R}_{i,j}\partial_x \delta\tilde{\bf  U}_j 

+ \sum_{j = 1}^{m+n} \sum_{k = 1}^{m+n}

\frac{ \partial  \mathbf{J_R}_{i,j}  }{\partial_x \tilde{\mathbf{U}}_k}  \delta\tilde{\bf  U}_k  \partial_x \tilde{\bf U}_j
 
 =  \\
 \sum_{k = 1}^{m+n}  
 \frac{ \partial  \tilde{\bf  L}_k }{ \partial  \bf \tilde{U}_k} \delta\tilde{\bf  U}_k \;.
 
\end{array}
\end{eqnarray}

If we denotes the inner product in $L^2(\R)$ by $\langle\cdot,\cdot\rangle$ and multiply last equation by functions $P_i$, such that $ lim_{|x|\rightarrow \infty} P_i(x,t) =  0$, $\forall \;t$, after integrating in space and time, we obtain
\begin{eqnarray}
\begin{array}{c}
\displaystyle \int_{0}^{T}
\langle \partial_t \delta\tilde{\bf U}_i ,P_i \rangle

 + \sum_{j = 1}^{m+n} 
\displaystyle \int_{0}^{T} 
\langle \partial_x \delta\tilde{\bf U}_j , \mathbf{J_R}_{i,j} P_i \rangle

\\

+ \sum_{j = 1}^{m+n} \sum_{k = 1}^{m+n}

\displaystyle \int_{0}^{T} 
\langle \delta\tilde{\bf U}_k ,  \frac{ \partial  \mathbf{J_R}_{i,j}  }{\partial \tilde{\mathbf{U}}_k} \partial_x \tilde{\bf U}_j P_i \rangle
 
 =  \\
 \sum_{k = 1}^{m+n}  
\displaystyle \int_{0}^{T} 
\langle  \delta\tilde{\bf U}_k ,   \frac{ \partial  \tilde{\bf  L}_k }{ \partial \bf \tilde{U}_k}  P_i \rangle \;.
 
\end{array}
\end{eqnarray}

Summing on indices $i$, yields
\begin{eqnarray}
\begin{array}{c}
\displaystyle \int_{0}^{T}
\sum_{i=1}^{m+n} 
\langle \partial_t \delta\tilde{\bf U}_i ,P_i \rangle

 - \sum_{j = 1}^{m+n} 
\displaystyle \int_{0}^{T} 
\langle  \delta\tilde{\bf U}_j ,  \sum_{i=1}^{m+n}  \partial_x (   \mathbf{J_R}_{i,j} P_i ) \rangle

\\

+ \sum_{j = 1}^{m+n} \sum_{k = 1}^{m+n}

\displaystyle \int_{0}^{T} 
\langle \delta\tilde{\bf U}_k ,  \sum_{i=1}^{m+n}  \frac{ \partial  \mathbf{J_R}_{i,j}  }{\partial \tilde{\mathbf{U}}_k}    \partial_x \tilde{\bf U}_j P_i \rangle
 
 =  \\
 \sum_{k = 1}^{m+n}  
\displaystyle \int_{0}^{T} 
\langle  \delta\tilde{\bf U}_k ,  \sum_{i=1}^{m+n}   \frac{ \partial  \tilde{\bf  L}_k }{ \partial \bf \tilde{U}_k}  P_i \rangle \;.
 
\end{array}
\end{eqnarray}

Then integrating by parts, we obtain
\begin{eqnarray}
\begin{array}{c}

\displaystyle 
\sum_{k=1}^{m+n} 
\langle \delta\tilde{\bf U}_k ,P_k \rangle |_{0}^{T}
-

\displaystyle \int_{0}^{T}
\sum_{k=1}^{m+n} 
\langle \delta\tilde{\bf U}_k , \partial_t P_k \rangle

 - \sum_{k = 1}^{m+n} 
\displaystyle \int_{0}^{T} 
\langle \delta\tilde{\bf U}_k ,  \sum_{i=1}^m      \mathbf{J_R}_{i,k}  \partial_x P_i  \rangle

\\

- \sum_{k = 1}^{m+n}
\displaystyle \int_{0}^{T} 
\langle \delta\tilde{\bf U}_k ,  \sum_{i=1}^m  \mathbf{D}_{k,j}    \partial_x \tilde{\bf U}_j \rangle

 = 
 \sum_{k = 1}^{m+n}  
\displaystyle \int_{0}^{T} 
\langle \delta\tilde{\bf U}_k ,  \sum_{i=1}^{m+n}   \frac{ \partial  \tilde{\bf  L}_k }{ \partial \bf \tilde{U}_k}  P_i \rangle \;,
 
\end{array}
\end{eqnarray}
where 
\begin{eqnarray}
\label{eq:defD}
\begin{array}{c}

\mathbf{D}_{k,j} =  \sum_{i=1}^{m+n}  \biggl( 
 \frac{ \partial  \mathbf{J_R}_{i,k}  }{\partial \tilde{\mathbf{U}}_j} 
 -
\frac{ \partial  \mathbf{J_R}_{i,j}  }{\partial \tilde{\mathbf{U}}_k} 
 \biggr) P_i \;.

\end{array}
\end{eqnarray}

On the other hand, $\delta J(\mathbf{b}) = \displaystyle  \int_{0}^{T} \int_{-\infty }^{\infty} (\psi(\mathbf{U},\mathbf{b}) - \bar{\psi}  ) ( \nabla_{\mathbf{U}} \psi ^T \delta \mathbf{U} + \nabla_{\mathbf{b}} \psi ^T \delta \mathbf{b} )$ or
$$
\delta J(\mathbf{b}) =  \sum_{k = 1}^{m+n}   \displaystyle \int_{0}^{T} 
\langle 
\delta\tilde{\bf U}_k , (\psi - \bar{\psi}  ) \frac{ \partial \psi }{ \partial \tilde{\bf  U}_k }
\rangle \;.
$$

Moreover
\begin{eqnarray}
\begin{array}{c}
\delta J =  \sum_{j=1}^{n}  \langle   \nabla J, \delta \mathbf{b} \rangle \;,
\end{array}
\end{eqnarray}
so, by identifying terms we obtain
\begin{eqnarray}
\label{FormalGrad:eq-1}
\begin{array}{c}
 \nabla J_j (x) = P_{m+j} (x,0) \;,
\end{array}
\end{eqnarray}
for all $j=1,...,n$. Imposing  $ P_i(x,T) = 0 $ for $i=1,...,n+m$ and $P_i(x,0) = 0 $ for $i=1,...,m$ we obtain the adjoint problem
\begin{eqnarray}
\begin{array}{c}
\partial_t P_k + \sum_{i=1}^{m+n}      \mathbf{J_R}_{i,k}  \partial_x P_i + \sum_{i=1}^{m+n}  D_{k,j}    \partial_x \tilde{\bf U}_j
=
\sum_{i=1}^{m+n}   \frac{ \partial  \tilde{\bf  L}_i }{ \partial \bf \tilde{U}_k}  P_i 
 +  (\psi - \bar{\psi}  ) \frac{ \partial \psi }{ \partial \tilde{\bf  U}_k } \;.
\end{array}
\end{eqnarray}

In a matrix form
\begin{eqnarray}
\label{eq:adjoint:1}
\begin{array}{c}
\partial_t \mathbf{P}  + \mathbf{J_R}^T \partial_x \mathbf{P}  =  \tilde{S}(\mathbf{ \tilde{U},P})
 \;, \\
 \mathbf{P}( x,T) = \mathbf{0} \;,
\end{array}
\end{eqnarray}
where $ \frac{1}{2}\nabla_{\tilde{\mathbf{U}}} ( (\psi - \bar{\psi}  )^2 )_{k} =  (\psi - \bar{\psi}  ) \frac{ \partial \psi }{ \partial \tilde{\bf  U}_k } $ and  $\mathbf{P} = [P_1,...,P_{n+m}]^T$. The source term in this case has the form
$$
 \tilde{S}(\mathbf{ \tilde{U},P}) = \frac{\partial \tilde{L}}{\partial \mathbf{\tilde{ U}}}^T \mathbf{P}  +  \frac{1}{2}\nabla_{\tilde{\mathbf{U}}} ( (\psi - \bar{\psi}  )^2 )  -  \mathbf{D} \partial_x \mathbf{\tilde{ U}} \;.
$$

In the sequel, we shall refer to $\mathbf{P}$ as {\it the adjoint state}.
Notice that  system (\ref{eq:adjoint:1}) is also hyperbolic, with the same eigenvalues that system (\ref{eq:adjoint:-2}).  It is because the Jacobian matrix of the system (\ref{eq:adjoint:1}) is $\mathbf{J}_R^T.$  Moreover, notice that system  (\ref{eq:adjoint:1}), henceforth referred to as {the adjoint model}, evolves back in time from $T$ to $0$.

\subsection{The unified formulation}

In this section, we construct a model system which accounts for both the PDE constraint and the adjoint model, aimed to construct a numerical method, able to solve simultaneously both systems including the treatment of non-conservative products. So, let us consider the following unified primal-dual system
\begin{eqnarray}
\label{eq:unifiedsystem:1}
\begin{array}{c}
\partial_t \mathbf{Q} + \partial_x \mathbf{F}(\mathbf{Q} ) + \mathbf{B}(\mathbf{Q}) \partial_x \mathbf{Q} = \mathbf{S} ( \mathbf{Q} ) \;,\\
\end{array}
\end{eqnarray}
where
\begin{eqnarray}
\begin{array}{c}
\mathbf{Q} = 
\left(  
\begin{array}{c}
\mathbf{\tilde{U} } \\
\mathbf{P}
\end{array}
\right)\;, \;

\mathbf{F}(\mathbf{Q} )= 
\left(  
\begin{array}{c}
\mathbf{\tilde{R} }(\mathbf{ \tilde{U} } ) \\
\mathbf{0}
\end{array}
\right) \;,
\\
\mathbf{B}(\mathbf{Q}) = 
\left(  
\begin{array}{cc}
\mathbf{0} & \mathbf{0} \\
\mathbf{D} & \mathbf{J_R}^T
\end{array}
\right) \;, \;

\mathbf{S} ( \mathbf{Q})
=
\left(  
\begin{array}{c}
\mathbf{ \tilde{L} } \\
- \frac{ \partial \mathbf{\tilde{L}}}{\partial \mathbf{\tilde{U} } }^T  \mathbf{P} 
+  \frac{1}{2}\nabla_{\tilde{\mathbf{U}}} ( (\psi - \bar{\psi}  )^2 )
\end{array}
\right)\;.
\end{array}
\end{eqnarray}
Notice that the state $\mathbf{Q}\in\R^{2(m+n)}$ contains the state variable $\mathbf{U}\in\R^m$, the model parameter $\mathbf{b}\in\R^n$ and the adjoint state $\mathbf{P}\in\R^{m+n}$. Moreover, it is expected that not only system (\ref{eq:unifiedsystem:1}) reproduces the evolution of the constraint PDE for marching forward in time, but also it approximates the evolution backward in time of the adjoint model. The aim is to design a numerical solver for the model (\ref{eq:unifiedsystem:1}) in order to solve constraint PDE and adjoint model with the same methodology and able to deal with non-conservative products.

We are going to prove that in most of the cases the unified system is hyperbolic if constraint PDE as well as the adjoint model, are hyperbolic. In such a case a finite volume scheme is a suitable method for solving this class of problems. Indeed, if (\ref{eq:primal:0}) admits a conservative formulation, there exists $\mathbf{\tilde{K}}(\mathbf{\tilde{U}})$ such that
\begin{eqnarray}
\begin{array}{c}
\mathbf{J_R} (\mathbf{\tilde{U}}) = \frac{ \partial  \mathbf{\tilde{K}}(\mathbf{\tilde{U}}) }{ \partial \mathbf{\tilde{U}} } \;.
\end{array}
\end{eqnarray}

Then $\mathbf{D} = \mathbf{0}$ and (\ref{eq:unifiedsystem:1}) is hyperbolic with the same eigenvalues that (\ref{eq:primal:2}). This is stated in the following
\begin{proposition}
\label{Prop:unifiedsystem:1}
If there exists a differentiable function $\mathbf{\tilde{K}}(\mathbf{\tilde{U}})$  such that   $\mathbf{J_R} (\mathbf{\tilde{U}}) = \frac{ \partial  \mathbf{\tilde{K}}(\mathbf{\tilde{U}}) }{ \partial \mathbf{\tilde{U}} } \;,$ then  $\mathbf{D} = \mathbf{0}$ and (\ref{eq:unifiedsystem:1}) is hyperbolic with eigenvalues 
$$
\{ \lambda_1,....,\lambda_{m+n} \} \;,
$$
each one with multiplicity $2$ and eigenvectors 
$$
\{ [\mathbf{v}_1, \mathbf{0}_{m+n} ]^T, [\mathbf{0}_{m+n}, \mathbf{v}^t_1 ]^T,....,[\mathbf{v}_{m+n}, \mathbf{0}_{m+n} ]^T, [\mathbf{0}_{m+n} ,\mathbf{v}^t_{m+n} ]^T \}\;, 
$$
where $\mathbf{v}_i$ is the eigenvector of $\mathbf{J_R}$ associated with $\lambda_i$ and $\mathbf{v}^t_i$ is the eigenvector of $\mathbf{J_R}^T$ associated with $\lambda_i$. Here $\mathbf{0}_{m+n} $ is the zero vector of $\mathbb{R}^{m+n}.$
\end{proposition}
\begin{proof}
Let us write system (\ref{eq:unifiedsystem:1}) in the quasilinear form
\begin{eqnarray}
\label{eq:unifiedsystem:ql-1}
\begin{array}{c}
\partial_t \mathbf{Q} + \mathbf{A}(\mathbf{Q}) \partial_x \mathbf{Q} = \mathbf{S} ( \mathbf{Q} ) \;,\\
\end{array}
\end{eqnarray}
where 
\begin{eqnarray}
\begin{array}{c}

\mathbf{A}(\mathbf{Q}) = 
\left(  
\begin{array}{cc}
\mathbf{\mathbf{J_R}} & \mathbf{0} \\
\mathbf{D} & \mathbf{J_R}^T
\end{array}
\right) \;.
\end{array}
\end{eqnarray}

Notice that $ det ( \mathbf{A} - \mu \mathbf{I}_{2(m+n)}  ) = det ( \mathbf{J_R} - \mu \mathbf{I}_{(m+n)}  ) det ( \mathbf{J_R}^T - \mu \mathbf{I}_{(m+n)}  ) = p(\mu)^2 ,$ where $p(\mu)$ is the characteristic polynomial of $\mathbf{J_R}$, from which we obtain that eigenvalues of this system are the same of $\mathbf{J_R}$, each one with at lest an algebraic multiplicity $2$.  It is because it may exists indices $i$ and $j$ such that  $i\neq j$ and $\lambda_i = \lambda_j$.

On the other hand, from the existence of $\mathbf{\tilde{K}}$, we note that 
\begin{eqnarray}
\begin{array}{c}

 \frac{ \partial  \mathbf{J_R}_{i,k}  }{\partial \tilde{\mathbf{U}}_j} = 
 \frac{ \partial^2  \mathbf{\tilde{K}}(\mathbf{\tilde{U}}) }{ \partial \mathbf{\tilde{U}}_k \partial \mathbf{\tilde{U}}_j } 
 =
 \frac{ \partial^2  \mathbf{\tilde{K}}(\mathbf{\tilde{U}}) }{ \partial \mathbf{\tilde{U}}_j \partial \mathbf{\tilde{U}}_k } 
=
\frac{ \partial  \mathbf{J_R}_{i,j}  }{\partial \tilde{\mathbf{U}}_k} 
\;,
\end{array}
\end{eqnarray}
so from (\ref{eq:defD}) we have $ \mathbf{D} = \mathbf{0}.$  Thus the Jacobian matrix of the system now takes the form
\begin{eqnarray}
\begin{array}{c}
\mathbf{A}(\mathbf{Q} ) =
\left[
\begin{array}{cc}
\mathbf{J_R} & \mathbf{0} \\
\mathbf{0}   & \mathbf{J_R}^T
\end{array}
\right]
\;.
\end{array}
\end{eqnarray}

In order to find the set of eigenvectors, let us denote by $\mathbf{v}_i$  the eigenvector of $\mathbf{J_R}$ associated with $\lambda_i$
and  $\mathbf{v}^t_i$ by the eigenvector of $\mathbf{J_R}^T$ associated with $\lambda_i$. Then is straightforward to verify that  
\begin{eqnarray}
\begin{array}{c}

\left[
\begin{array}{cc}
\mathbf{J_R} & \mathbf{0} \\
\mathbf{0}   & \mathbf{J_R}^T
\end{array}
\right]
\left[
\begin{array}{c}
\mathbf{v}_i
\\
\mathbf{0}_{n+m}
\end{array}
\right]
=
\lambda_i
\left[
\begin{array}{c}
 \mathbf{v}_i
\\
\mathbf{0}_{n+m}
\end{array}
\right]
\end{array}
\end{eqnarray}
and 
\begin{eqnarray}
\begin{array}{c}

\left[
\begin{array}{cc}
\mathbf{J_R} & \mathbf{0} \\
\mathbf{0}   & \mathbf{J_R}^T
\end{array}
\right]
\left[
\begin{array}{c}
\mathbf{0}_{n+m}
\\
 \mathbf{v}^t_i
\end{array}
\right]

=
\lambda_i
\left[
\begin{array}{c}
\mathbf{0}_{n+m}
\\
 \mathbf{v}^t_i
\end{array}
\right]\;.
\end{array}
\end{eqnarray}

In this form, a set of eigenvectors is obtained. The verification of the linear independence comes from the fact that  $\{\mathbf{v}_1, ..., \mathbf{v}_m\}$ as well as $\{\mathbf{v}^t_1, ..., \mathbf{v}^t_m\}$ are two sets of linearly independent vectors. 
 \end{proof}
%
%
\begin{remark}
The previous proposition shows that unified systems coming from conservative hyperbolic  equations are always hyperbolic. This does not mean that non-conservative systems are not hyperbolic. So in the case of systems where $\mathbf{D}$ is not the null matrix, the analysis of hyperbolicity has to be done case by case.  
\end{remark}

\subsection{Existence of solutions, the symetrizable case}
\label{ExSol}

In this section we discuss the setting where a well-posedness theorem for system \eqref{eq:unifiedsystem:1} can be formulated. Recall from Proposition \ref{Prop:unifiedsystem:1}, system \eqref{eq:unifiedsystem:1} can be written in the quasilinear form
\begin{eqnarray}
\label{eq:existence:1}
\begin{array}{c}
\partial_t \mathbf{Q} + \mathbf{A}(\mathbf{Q}) \partial_x \mathbf{Q} = \mathbf{S} ( \mathbf{Q} ) \;,\\
\end{array}
\end{eqnarray}
with
\begin{eqnarray}
\begin{array}{c}

\mathbf{A}(\mathbf{Q}) = 
\left(  
\begin{array}{cc}
\mathbf{\mathbf{J_R}} & \mathbf{0} \\
\mathbf{D} & \mathbf{J_R}^T
\end{array}
\right),

\quad
\mathbf{S} ( \mathbf{Q})
=
\left(  
\begin{array}{c}
\mathbf{ \tilde{L} } \\
- \frac{ \partial \mathbf{\tilde{L}}}{\partial \mathbf{\tilde{U} } }^T  \mathbf{P} 
+  \frac{1}{2}\nabla_{\tilde{\mathbf{U}}} ( (\psi - \bar{\psi}  )^2 )
\end{array}
\right).
\end{array}
\end{eqnarray}

In \cite{taylor2010partial}, has been made the hypothesis of strict hyperbolicity: that the Jacobian matrix $\mathbf{A}$ has real and distinct eigenvalues. Then, under this assumption, \eqref{eq:existence:1} is called a symetrizable hyperbolic system provided there exists $\mathbf{A}_0(\mathbf{Q})$ positive-definite, such that $\mathbf{A}_0\mathbf{A}$ is symmetric. See Proposition 2.2, Chapter 16 in \cite{taylor2010partial} for further details.


The assumption of the symmetry of $\mathbf{A}$ is natural in this context of conservation laws. In \cite{lax1954weak}, Lax introduced a general notion of symmetrizer in the context of pseudodifferential operators to prove that any strictly hyperbolic system is symmetrizable. Also, as was mentioned in \cite{serre2015relative} under the context of systems of conservation laws having an entropy solution, Godunov \cite{Godunov:1961a} and Friedrichs and Lax \cite{Friedrichs:1971a} observed that systems of conservation laws admitting a strongly convex entropy can be symmetrized. It is then under this context of strictly hyperbolic matrices that we state the following theorem for the system \eqref{eq:existence:1}. This theorem is the key point to set a descent method to solve optimal control problems related to conservation laws. Even though its proof is classical, see \cite{taylor2010partial}, we outline some key steps for the sake of completeness. Let $I=(-a,b)$ be an interval with $t\in I$.

\begin{theorem}
\label{Theorem1}
Let $\mathbf{Q}(x,0)=f\in H^k(\R^d)$, $k>\frac{d}{2}+1$. If $\mathbf{A}_0(t,x,\mathbf{Q})$, $\mathbf{A}(t,x,\mathbf{Q})$ in $C(I,H^k\times H^k(\R^d))$ and $\mathbf{L}(\mathbf{Q},\mathbf{b})$, $\bar{\psi}$, $\psi(\mathbf{Q},\mathbf{b})$ in $C(I,H^{k+1}\times H^{k}(\R^d))$, then the initial value problem \eqref{eq:existence:1} has a unique solution in $C(I,H^k\times H^k(\R^d))$.
\end{theorem}

\begin{proof}
Let $\{J_\epsilon:0<\epsilon\le1\}$ a Friedrichs mollifier. The idea is to obtain a solution to \eqref{eq:existence:1} as the limit of unique solutions $\mathbf{Q}_\epsilon$ to the systems of ODEs
\begin{equation}
\label{eq:existence:2}
\left\{
\begin{aligned}
& \mathbf{A}_0(t,x,J_\epsilon \mathbf{Q}_\epsilon)\partial_t\mathbf{Q}_\epsilon+J_\epsilon \mathbf{A}(t,x,J_\epsilon \mathbf{Q}_\epsilon)\partial_xJ_\epsilon \mathbf{Q}_\epsilon=J_\epsilon \mathbf{S}(t,x,J_\epsilon \mathbf{Q}_\epsilon), \\
& \mathbf{Q}_\epsilon(0)=f,
\end{aligned}
\right.
\end{equation}
for $t$ close to $0$. We denote $\mathbf{A}_\epsilon=\mathbf{A}(t,x,J_\epsilon \mathbf{Q}_\epsilon)$.

Let us consider the $L^2$-inner product
\[ (w,\mathbf{A}_{0\epsilon} w), \quad \mathbf{A}_{0\epsilon}=\mathbf{A}_0(t,x,J_\epsilon \mathbf{Q}), \]
which, by the positivity of $\mathbf{A}_0$, defines an equivalent $L^2$-norm. Therefore by the symmetry of $\mathbf{A}_0$ and \eqref{eq:existence:2}
\begin{align*}
\frac{d}{dt}(D^\alpha \mathbf{Q}_\epsilon,\mathbf{A}_{0\epsilon}(t)D^\alpha \mathbf{Q}_\epsilon)&=-2(D^\alpha(J_\epsilon \mathbf{A}_{\epsilon}\partial_xJ_\epsilon \mathbf{Q}_\epsilon),D^\alpha \mathbf{Q}_\epsilon)+2(D^\alpha J_\epsilon \mathbf{S},D^\alpha \mathbf{Q}_\epsilon) \\
&\phantom{=} -2([D^\alpha,\mathbf{A}_{0\epsilon}]\partial_t\mathbf{Q}_\epsilon,D^\alpha \mathbf{Q}_\epsilon)+(D^\alpha \mathbf{Q}_\epsilon,\mathbf{A}_{0\epsilon}'(t)D^\alpha \mathbf{Q}_\epsilon).
\end{align*}

Then, using Moser--type \cite{moser1966rapidly} estimates and Sobolev embeddings, for $|\alpha|\le k$, one has
\begin{equation}
\label{existence3}
\frac{d}{dt}(D^\alpha \mathbf{Q}_\epsilon,\mathbf{A}_{0\epsilon}(t)D^\alpha \mathbf{Q}_\epsilon)\le C_k(\|\mathbf{Q}_\epsilon(t)\|_{C^1})(1+\|J_\epsilon \mathbf{Q}_\epsilon(t)\|_{H^k}^2).
\end{equation}

Since $\|J_\epsilon \mathbf{Q}_\epsilon(t)\|_{H^k}\le C\|\mathbf{Q}_\epsilon(t)\|_{H^k}$, by \eqref{existence3}, $\mathbf{Q}_\epsilon$ satisfies

\begin{equation*}
\left\{
\begin{aligned}
& \frac{d}{dt}\|\mathbf{Q}_\epsilon(t)\|_{H^k}^2\le E(\|\mathbf{Q}_\epsilon(t)\|_{H^k}^2), \\
& \|\mathbf{Q}_\epsilon(0)\|_{H^k}^2=\|f\|_{H^k}^2.
\end{aligned}
\right.
\end{equation*}

Thus by Gronwall's inequality, $\mathbf{Q}_\epsilon$ exists for $t$ in an interval $[0,b)$ independent of $\epsilon$ and
\begin{equation}
\label{existence4}
\|\mathbf{Q}_\epsilon(t)\|_{H^k}\le K(t)<\infty.
\end{equation}

Moreover, by the time-reversibility, $\mathbf{Q}_\epsilon$ exists and is bounded on $I=(-a,b)$.

By \eqref{existence4} and \eqref{eq:existence:2}, $\mathbf{Q}_\epsilon\in C(I,H^k)\cap C^1(I,H^{k-1})$ is bounded and therefore will have a weak limit point $\mathbf{Q}\in L^\infty(I,H^k(\R^d))\cap\lip(I,H^{k-1}(\R^d))$. Moreover, by Ascoli's Theorem there is a sequence, still denoted $\mathbf{Q}_\epsilon$, such that $\mathbf{Q}_\epsilon\rightarrow \mathbf{Q}$ in $C(I,H^{k-1}(\R^d))$ and since $k>\frac{d}{2}+1$
\begin{equation*}
\mathbf{Q}_\epsilon\rightarrow \mathbf{Q} \quad \text{in} \ C(I,C^1(\R^d)).
\end{equation*}

Consequently, \eqref{eq:existence:1} follows in the limit from \eqref{eq:existence:2}. Uniqueness is treated in a similar way to the estimates above. This conclude the proof of the Theorem.
\end{proof}

\begin{corollary}
\label{Cor1}
Let $\mathbf{b}\in H^k(\R^d)$, $k>\frac{d}{2}+1$. If $\mathbf{Q}\in C(I,H^k(\R^d))$ is the solution of \eqref{eq:existence:1} with initial condition $\mathbf{Q}(0)=\mathbf{b}$ then there exists an open neighborhood $V_\mathbf{b}$ of $\mathbf{b}$ in $H^k(\R^d)$ such that system \eqref{eq:existence:1} with initial condition $\beta\in V_\mathbf{b}$ and right hand side $\mathbf{g}=\partial_\mathbf{Q}\mathbf{S}-\partial_\mathbf{Q}J_R\partial_x\mathbf{Q}\in C(H^k(\R^d))$ have a solution $\mathbf{Q}_\beta$. Moreover the mapping $G:V_\mathbf{b}\rightarrow C(H^k(\R^d))$, defined by $G(\beta)=\mathbf{Q}_\beta$, is of class $C^\infty$. Finally, if $V=DG(\beta)\cdot h$, for some $\beta\in V_\mathbf{b}$ and some $h\in H^k(\R^d)$, then $V$ is the unique solution of the problem
\[
\partial_t V + J_R(\mathbf{Q}_\beta) \partial_x V = \mathbf{g}V, \quad V(0)=\mathbf{b}.
\]
\end{corollary}
\begin{proof}
Proceeding as in the proof of Theorem \ref{Theorem1}, we can consider
\[ F:C(H^k(\R^d))\times H^k(\R^d)\rightarrow C(H^k(\R^d))\times H^k(\R^d) \]
the mapping given by
\[ F(\mathbf{Q},\mathbf{b})=(\partial_t \mathbf{Q}+\mathbf{A}(\mathbf{Q}) \partial_x\mathbf{Q}-\mathbf{S}(\mathbf{Q}),\mathbf{Q}(0)-\mathbf{b}). \]

Then $F$ is of class $C^\infty$ and
\[ \frac{\partial F}{\partial\mathbf{Q}}(\mathbf{Q},\mathbf{b})\cdot V=(\partial_t V + J_R(\mathbf{Q}_\beta) \partial_x V-\mathbf{g}V,V(0)). \]

Applying Theorem \ref{Theorem1} we have that $\frac{\partial F}{\partial\mathbf{Q}}(\mathbf{Q,\mathbf{b}})$ is an isomorphism from $C(H^k(\R^n))$ onto $C(H^k(\R^n))\times H^k(\R^n)$. Now, if $\mathbf{Q}$ is a solution of \eqref{eq:existence:1}, then $F(\mathbf{Q},\mathbf{b})=(0,0)$. Therefore, by the implicit function theorem, there exists an open neighborhood $V_\mathbf{b}$ of $\mathbf{b}$ in $H^k(\R^d)$ and $G:V_\mathbf{b}\rightarrow C(H^k(\R^d))$ such that $F(G(\beta),\beta)=(0,0)$ for every $\beta\in V_\mathbf{b}$. Moreover, $G$ is also of class $C^\infty$ and
\[ \frac{\partial F}{\partial \mathbf{Q}}(\mathbf{Q}_\beta,\beta)\circ DG(\beta)\cdot h+\frac{\partial F}{\partial \beta}(\mathbf{Q}_\beta,\beta)\cdot h=(0,0) \quad \forall h\in H^k(\R^d). \]

Then, if we set $V=DG(\beta)\cdot h$, we get
\[
\partial_t V + J_R(\mathbf{Q}_\beta) \partial_x V = \mathbf{g}V, \quad V(0)=\mathbf{b}.
\]
\end{proof}

Note that we have established the results above for $x\in\R^d$ and $k>\frac{d}{2}+1$. This is to keep the generality of this results. In this work will be enough to consider $d=1$ and $k>3/2$ as we did in the formulation of system \eqref{eq:primal:0}.

\subsection{The optimal control problem}
\label{OptCont}

One of our main motivations applying an optimal control strategy to general conservation laws is the possibility of study numerical inverse problems related to this type of systems. Example of them are the bottom detection in water-waves type system, see \cite{Zuazua:2015a}, or the tracking of initial conditions in the presence of shocks \cite{Castro:2008a-1,Lecaros:2014a}, to mention but a few.

We are considering the functional
\begin{equation*}
J(\mathbf{Q},\mathbf{b})=\frac{1}{2}\int_0^T\int_{\R}|\psi(\mathbf{Q},\mathbf{b})-\bar{\psi}|^2dxdt+\frac{\alpha}{2}\|\mathbf{b}\|_{H^k}^2,
\end{equation*}
with $\bar{\psi}(t,x)\in C((0,T),H^k(\R^d))$, $k>\frac{d}{2}+1$, a given function and $\alpha>0$. We want to find $\mathbf{b}$ such that $J$ is small. The nature of $\psi$ is not specified and it depends on the physical problem under consideration, as well as the itself election of $J$. The admissible control set is a nonempty convex closed subset $U_{ad}$ of $H^k(\R^d)$.

Then the optimal control problem is formulated in the following way

\[
\label{optProblem}
\left.
\begin{array}{c}
\text{Minimize} \ J(\mathbf{Q},\mathbf{b}),\\
(\mathbf{Q},\mathbf{b})\in H^1(H^k(\R^d))\times U_{ad} \ \text{verifies} \ \eqref{eq:existence:1}.
\end{array}
\right\}\tag{P}
\]

Now we study the existence of a solution of \eqref{optProblem}.

\begin{theorem}
\label{TheorExiMini}
Let $k>\frac{d}{2}+1$. If there exists a feasible pair $(\mathbf{Q},\mathbf{b})\in U_{ad}\times H^1(H^k(\R^d))$ satisfying \eqref{eq:existence:1}, then there exists at least one optimal solution $(\mathbf{Q}_0,\mathbf{b}_0)$ of \eqref{optProblem}.
\end{theorem}
\begin{proof}
Since there is one feasible pair for the problem and $J$ is bounded below by zero we may take a minimizing sequence $\{(\mathbf{Q}_n,\mathbf{b}_n)\}\subset U_{ad}\times H^1(H^k(\R^d))$. Then $\frac{\alpha}{2}\|\mathbf{b}_n\|_{H^k}^2\le J(\mathbf{Q}_n,\mathbf{b}_n)<\infty$, which implies that $\{\mathbf{b}_n\}$ is a bounded sequence in $H^k(\R^d)$. Thus, we can take a subsequence, denoted in the same way, such that $\mathbf{b}_n\rightarrow\mathbf{b}_0$ weakly in $H^k(\R^d)$. Since $U_{ad}$ is closed and convex, $\mathbf{b}_0\in U_{ad}$.
On the other hand, following the proof of Theorem \ref{Theorem1}, we have that $\{\mathbf{Q}_n\}$ is bounded in $H^1(H^k(\R^d))$. Therefore we can assume, by taking a subsequence if necessary, that $\mathbf{Q}_n\rightarrow\mathbf{Q}_0$ weakly in $H^1(H^k(\R^d))$. Using the continuity of the inclusion $H^1(H^k(\R^d))\subset C(H^{k-1}(\R^d))$ we can pass to the limit in system \eqref{eq:existence:1} satisfied by $(\mathbf{Q}_n,\mathbf{b}_n)$ and to conclude that $(\mathbf{Q}_0,\mathbf{Q}_0)$ also satisfies \eqref{eq:existence:1}. Namely, $(\mathbf{Q}_0,\mathbf{b}_0)$ is a feasible pair for problem \eqref{optProblem}. Taking into consideration that $J(\mathbf{Q},\mathbf{b})$ is
weakly lower semi-continuous, the result follows.
\end{proof}

Let us remark that the assumption of Theorem \ref{TheorExiMini} can be checked by taking $\mathbf{Q}\in H^1(H^k(\R^d))$ with $\mathbf{Q}(0)=f$. This uniquely defines $\mathbf{b}$ from the partial differential equation \eqref{eq:existence:1} as $\mathbf{b}=\mathbf{Q}(0)$. If $\mathbf{b}\in U_{ad}$ then the assumption is satisfied. This is the case if $U_{ad}=H^k(\R^d)$.

Finally we state the following proposition giving the optimality conditions satisfied by the solutions $(\mathbf{Q}_0,\mathbf{b}_0)$ of \eqref{optProblem}. Its proof is based on Corollary \ref{Cor1} and the results obtained in Section \ref{Section2.1}.

\begin{proposition}
Let $k>\frac{d}{2}+1$ and assume that $(\mathbf{Q}_0,\mathbf{b}_0)$ is a solution of \eqref{optProblem}. Then there exist a unique element $\mathbf{P}_0\in C(H^k(\R^d))$ such that the following system is satisfied
\begin{eqnarray*}
\begin{array}{c}
\partial_t \mathbf{Q}_0 + \mathbf{J}_R \partial_x \mathbf{Q}_0 = \mathbf{L} ( \mathbf{Q}_0 )
 \;, \\
 \mathbf{Q}_0( x,T) = \mathbf{b}_0 \;,
\end{array}
\end{eqnarray*}
\begin{eqnarray*}
\begin{array}{c}
\partial_t \mathbf{P}_0  + \mathbf{J}_R^T \partial_x \mathbf{P}_0  = \frac{\partial \mathbf{L}}{\partial \mathbf{Q_0}}^T \mathbf{P}_0  +  \frac{1}{2}\nabla_{\mathbf{Q}_0} ( (\psi - \bar{\psi}  )^2 )  -  \mathbf{D} \partial_x \mathbf{Q_0}
 \;, \\
 \mathbf{P}_0( x,T) = \mathbf{0} \;,
\end{array}
\end{eqnarray*}
\begin{eqnarray*}
\begin{array}{c}
\mathbf{P}_0|_{t=\mathbf{0}}.
\end{array}
\end{eqnarray*}
\end{proposition}

\section{The numerical scheme}\label{section:numericalmethod}
In this section, we present the scheme for solving simultaneously  conservative as well as non-conservative hyperbolic balance laws of the form (\ref{eq:unifiedsystem:1}). Cell averages are evolved by following the one-step finite volume formula
\begin{eqnarray}
\label{one-step:1}
\begin{array}{lcl}
\mathbf{Q}_i^{n+1} &=&  \mathbf{Q}_i^{n} - \eta \frac{\Delta t}{\Delta x} (\mathbf{F}_{i+\frac{1}{2}}-\mathbf{F}_{i-\frac{1}{2}})- \eta\frac{\Delta t}{\Delta x} (\mathbf{D}_{i+\frac{1}{2}} + \mathbf{D}_{i-\frac{1}{2}} ) \\
& &- \frac{\eta}{\Delta x} (\mathbf{B}\partial \mathbf{Q})_i + \eta \Delta t \mathbf{S}_i \;, 
\end{array}
\end{eqnarray}
where $\eta $ controls the evolution of the system. This is because in the context of the adjoint method we need the constraint PDE evolves forward in time up to $T$ and then the adjoint model evolves backward in time from $T$ to $0$. So, $\eta = 1$ makes the forward in time and $\eta = -1 $ makes the backward in time. 
This formula is a conventional one, available in the literature for solving non-conservative schemes, see for instance \cite{Dumbser:2014a}.  However, we are going to modify the usual implementation of these type of schemes in order to account for the correct wave propagation. By the sake of completeness, we provide a detailed description of the scheme. We have
\begin{eqnarray}
\label{integrals:eq-1}
\begin{array}{c}
( \mathbf{B}\partial \mathbf{Q})_i = \frac{1}{\Delta t \Delta x} \displaystyle \int_{t}^{ t^{n+1} }
\displaystyle \int_{x_{ i-\frac{1}{2} }}^{ x_{ i+\frac{1}{2} } }
\mathbf{B}( \mathbf{Q}_i (\xi,\tau)) \partial_x  \mathbf{Q}_i(\xi,\tau) d\xi d\tau \;, 

\\

\mathbf{S}_i = \frac{1}{\Delta t \Delta x} \displaystyle \int_{t}^{t^{n+1}}
\displaystyle \int_{ x_{ i-\frac{1}{2} }}^{ x_{ i+\frac{1}{2} } }
\mathbf{S}( \mathbf{Q}_i (\xi,\tau)) d\xi d\tau  \;
\end{array}
\end{eqnarray}
and
\begin{eqnarray}
\label{integrals:eq-2}
\begin{array}{c}
\mathbf{\mathbf{F}}_{i+\frac{1}{2}} =  \frac{1}{\Delta t} \int_{t^n}^{t^{n+1}} 
\tilde{\mathbf{F}}_{h}( \mathbf{Q}_{i+\frac{1}{2}}^{-}(\tau ), \mathbf{Q}_{i+\frac{1}{2}}^{+}(\tau ) ) d\tau \;,  
\\
\mathbf{D}_{i+\frac{1}{2}} =   \frac{1}{\Delta t} \int_{t^{n}}^{t^{n+1}} 

\tilde{  \bf D}(  \mathbf{Q}_{i+\frac{1}{2}}^{-}(\tau ), \mathbf{Q}_{i+\frac{1}{2}}^{+}(\tau ) ) d\tau \;,

\end{array}
\end{eqnarray}
where
\begin{eqnarray}
\label{DOT:eq-1}
\begin{array}{c}

\tilde{\mathbf{F}}_{h}( \mathbf{Q}_{h}^{-}(\tau ), \mathbf{Q}_{h}^{+}(\tau ) )
= \\
\frac{1}{2}(
\mathbf{F}( \mathbf{Q}_{h}^{-} ) + (\mathbf{F}( \mathbf{Q}_{h}^{+}  ) ) 
- \eta \displaystyle \frac{1}{2} \int_{0}^{1} 

| \mathbf{A}(\Psi(s; \mathbf{Q}_{h}^{-},\mathbf{Q}_{h}^{+} )) |     \frac{d\Psi}{ds} ds \;.
\end{array}
\end{eqnarray}

Moreover  $\mathbf{A} := \frac{\partial \mathbf{F}}{\partial \mathbf{Q}}$, with $|\mathbf{A}| = \mathbf{R} |\Lambda|\mathbf{R}^{-1} $, where $|\Lambda| = diag( |\lambda_1|,...,|\lambda_m|)$ and then numerical  jumps are obtained as  
\begin{eqnarray}
\begin{array}{c}

\tilde{  \bf D}(  \mathbf{Q}_{h}^{-}(\tau ), \mathbf{Q}_{h}^{+}(\tau ) ) = 
\eta \cdot \frac{1}{2} \displaystyle \int_{0}^{1} 

\mathbf{B}( \Psi(s; \mathbf{Q}_{h}^{-}, \mathbf{Q}_{h}^{+} )   
\frac{d\Psi}{ds} ds \;,
\end{array}
\end{eqnarray}
with
\begin{eqnarray}
\begin{array}{c}
\Psi(s; \mathbf{Q}_{h}^{-}, \mathbf{Q}_{h}^{+} )  :=  \mathbf{Q}_{h}^{-} + s(  \mathbf{Q}_{h}^{+} -  \mathbf{Q}_{h}^{-}  ) \;.
\end{array}
\end{eqnarray}

Equation (\ref{DOT:eq-1}) corresponds to the DOT (Dumbser-Osher-Toro) Riemann solver, introduced in \cite{Dumbser:2010b}. The integrals are evaluated numerically, by using some quadrature rule. Here $\mathbf{Q}_{i}(\xi,\tau)$ is a predictor inside the computational cell $[x_{i-\frac{1}{2}}, x_{i+\frac{1}{2}}]$ and $ \mathbf{Q}_{i+\frac{1}{2} }^{\pm} (\tau)$ are extrapolations of the solution at both sides of the cell interface position $x = x_{i+\frac{1}{2} }$. They are usually computed by using the so-called Cauchy-Kowalewsky procedure and Taylor series expansions.

\subsection{The predictor step for second order of accuracy}
In this section, we deal with the strategy to get predictor inside the computational cell. This stage requires two ingredients, the first one is a polynomial representation of the solution and then a local evolution in time of approximate values within the computational cell at located spatial positions.  The second stage is given by the evolution of values within the computational cell, the resulted state of the local evolution is commonly referred to as the predictor.

The polynomial representation of the solution is carried in terms of the so called {\it reconstruction procedure}, \cite{Harten:1987a}. 
Let us define the reconstruction polynomial within cell $[x_{i - \frac{1}{2}},  x_{i + \frac{1}{2}}]$ as
\begin{eqnarray}
\begin{array}{c}
\mathbf{P}_i(\xi) = \mathbf{Q}_i^n + (\xi - \frac{1}{2}) \mathbf{\Delta}_i\;,
\end{array}
\end{eqnarray}
with $ \mathbf{\Delta}_i$ being a slope endowing the scheme with the Total Variation Diminishing property. In this work we use the so-called Minmod limiter, see the pioneering work \cite{Toro:2009a} about TVD flux limiters for hyperbolic conservation laws.

In a component wise, it has the form
\begin{eqnarray}
\begin{array}{ccc}

\mathbf{\Delta}_{i,j} = \left\{
\begin{array}{ccl}
0   & ,& if\;  ( \mathbf{Q}_{i,j}^n - \mathbf{Q}_{i-1,j}^n ) \cdot ( \mathbf{Q}_{i+1,j}^n - \mathbf{Q}_{i,j}^n ) \leq 0 \;,  \\

( \mathbf{Q}_{i,j}^n - \mathbf{Q}_{i-1,j}^n )   & ,& if \;  | \mathbf{Q}_{i,j}^n - \mathbf{Q}_{i-1,j}^n | <  |( \mathbf{Q}_{i+1,j}^n - \mathbf{Q}_{i,j}^n ) | \;, \\

( \mathbf{Q}_{i+1,j}^n - \mathbf{Q}_{i,j}^n )   & ,& if \;  | \mathbf{Q}_{i+1,j}^n - \mathbf{Q}_{i,j}^n | <  |( \mathbf{Q}_{i,j}^n - \mathbf{Q}_{i-1,j}^n ) | \;. \\ 
 
\end{array}

\right. 

\end{array}
\end{eqnarray}

Polynomial $ \mathbf{P}_i$ is defined in terms of a local variable  $\xi\in[0,1]$ and it is related with the computational cell $[x_{i - \frac{1}{2}},  x_{i + \frac{1}{2}}]$ through out the relationship $x = x_{i - \frac{1}{2}} + \xi \Delta x$. 

The predictor is an approximation of the solution within the cell, which represents the local evolution of the solution at $x$. This evolution in time is given in terms of a Taylor series expansion
\begin{eqnarray}
\begin{array}{c}
\mathbf{Q}_i (\xi,\tau) = \mathbf{Q}(\xi,0_+) + \tau \partial_t \mathbf{Q} (\xi, 0_+)\;.
\end{array}
\end{eqnarray}

Notice that the series is truncated up to the first spatial derivative. This is enough to produce solutions of second order of accuracy.
The time derivative is expressed in terms of purely spatial derivatives. This is called the Cauchy-Kowalewsky procedure and it is  based on the use of the governing equation
\begin{eqnarray}
\begin{array}{c}
\tau \partial_t \mathbf{Q} (\xi, \tau)=  - \mathbf{A}(\mathbf{Q} (\xi, \tau)) \partial_x \mathbf{Q}(\xi,\tau)\;.
\end{array}
\end{eqnarray}

Replacing into the expansion, we have
\begin{eqnarray}
\begin{array}{c}\mathbf{Q}_i (\xi,\tau) = \mathbf{Q}(\xi,0_+)   - \tau \mathbf{A}(\mathbf{Q} (\xi, 0_+)) \partial_x \mathbf{Q}(\xi,0_+)\;.
\end{array}
\end{eqnarray}

Notice that to obtain the evolution at time $\tau$,  we require an approximation of the solution at local time $\tau =0$, which is associated with the global time level $ t^n$. As finite volume only provide an approximation of cell averages, the punctual information within cells is not straightforward obtained. In this stage we use the polynomial representation $\mathbf{P}_i(\xi)$ of the solution at time level $t^n$, $\mathbf{Q}_i^n$, so the Taylor series expansion takes the form
\begin{eqnarray}
\begin{array}{c}\mathbf{Q}_i (\xi,\tau) = \mathbf{P}_i(\xi)   - \tau \mathbf{A}(\mathbf{P}_ (\xi) ) \mathbf{P}'_i(\xi)\frac{1}{\Delta x}\;,
\end{array}
\end{eqnarray}
which is a polynomial in time. Using this polynomial, we evaluate the integrals (\ref{integrals:eq-1}) and (\ref{integrals:eq-2}). In particular, for the flux evaluation, we require
\begin{eqnarray}
\begin{array}{lcl}
\mathbf{Q}^{-}_{i+\frac{1}{2}}(\tau) &=& \mathbf{Q}_{i+1}(0,\tau) \;, \\
\mathbf{Q}^{+}_{i+\frac{1}{2}}(\tau) &=& \mathbf{Q}_{i}(1,\tau) \;.
\end{array}
\end{eqnarray}

\subsection{Treatment for forward and backward evolution in time}\label{sect:variation}
The proposed method uses approximations of classical Riemann problem at cell interfaces, for computing the integrals in (\ref{integrals:eq-2}). This is carried out for both the forward ($\eta= 1$) and backward ($\eta= -1$) evolution in time, as hyperbolic equations are not reversible in general we may obtain different results and then we must select the right value for the minimization problem.
To illustrate this fact, let us consider the Burgers equation in the context of the unified system. So, the governing equation associated to the constraint PDE, has the form 
\begin{eqnarray}
\begin{array}{c}
\partial_t q + \eta q\partial_x q = 0 \;,
\end{array}
\end{eqnarray}
where $\eta =1 $ for forward evolution in time and $\eta = -1$ for backward evolution in time.  If we study Riemann problems associated to the unified system but restricted to the state variable, we have
\begin{eqnarray}
\begin{array}{c}
\partial_t q + \eta q\partial_x q = 0 \;, t > 0\;,\\
q(x,0)= \left\{
\begin{array}{c}
q_L\;, x<0\;, \\
q_R\;, x\geq 0 \;,\\
\end{array}
\right.
\end{array}
\end{eqnarray}
with $q_L$ and $q_R$ constant states. The Riemann problem solution of this problem is well studied in the literature, see for instance \cite{Toro:2009a}. Let us assume that $q_L > 0 > q_R$ and $\eta = 1$, so the solution at the interface position is a shock wave propagating with velocity $S = \frac{q_L +q_R}{2}< q_L$. Therefore, at the interface position it has the value $q(0,t) = q_L$.  If now, we solve backward in time the unified problem, we set $\eta = -1$, so we have  
\begin{eqnarray}
\begin{array}{c}
\partial_t q - q\partial_x q = 0 \;,\\
q(x,0)= \left\{
\begin{array}{c}
q_L\;, x<0\;, \\
q_R\;, x\geq 0 \;.\\
\end{array}
\right.
\end{array}
\end{eqnarray}

Then, by using the change of variable $x=-x$, we obtain
\begin{eqnarray}
\begin{array}{c}
\partial_t q + q\partial_x q = 0 \;,\\
q(x,0)= \left\{
\begin{array}{c}
q_R\;, x < 0\;, \\
q_L\;, x \geq 0 \;,\\
\end{array}
\right.
\end{array}
\end{eqnarray}
and so the solution corresponds to a rarefaction wave and thus, at the interface position, takes the value $q(0,t) = 0$. 

On the other hand, the adjoint  state in the backward evolution involves $q$ but it requires the value generated in the forward evolution, in this particular case the shock wave and not the right value of $q$ generated in the backward evolution (a rarefaction wave). If we choose the rarefaction wave, then we sub estimate the required values and thus it may have an impact in the minimization procedure.

To force the adjoint state to take the wave propagation of interest for us, we record the state variables from the solution obtained in the forward evolution in time and then use it to provide the evolution of the adjoint state in the backward evolution. That means for $\eta = -1$ we only use (\ref{one-step:1}) to evolve the variables associated to the adjoint states and we froze the evolution of state variables.  The suggested application of the strategy is listed below.
\begin{itemize}
\item[1)] Given $\mathbf{b}^k$, we solve the constraint PDE, using the one-step finite volume formula (\ref{one-step:1}) with $\eta = 1$ up to the output time $t = T$ and record from $\mathbf{Q}$ the variables associated to the state variables $\mathbf{U}$. Final time $T$ is reached in a finite number of steps, let say $n_T$. That means, starting from $t^0 = 0$, we do $t^{n+1} = t^{n} + \Delta t$,  $n_T$ times.

\item[2)] Solve the adjoint problem. We use the same formula (\ref{one-step:1}) with $\eta=-1$ and the initial condition $\mathbf{Q}^{n_T}_{j+m+n} = \mathbf{P}_{j}^n = 0$, but applied in a different way as in step 1). We turn off the evolution  associated with the state variables. That means, we do
\begin{eqnarray}
\label{one-step:1-1}
\begin{array}{c}
\mathbf{P}_{i,j}^{n} =  \mathbf{P}_{i,j}^{n+1}+ \frac{\Delta t}{\Delta x} (\mathbf{F}_{i+\frac{1}{2},j+m+n}-\mathbf{F}_{i-\frac{1}{2},j+m+n}) \\ 
+\frac{\Delta t}{\Delta x} (\mathbf{D}_{i+\frac{1}{2},j+m+n} + \mathbf{D}_{i-\frac{1}{2},j+m+n} ) + \frac{1}{\Delta x} (\mathbf{B}\partial \mathbf{Q})_{i,j+m+n} - \Delta t \mathbf{S}_{i,j+m+n} \;, 
\end{array}
\end{eqnarray}
for $j=1,...,(m+n)$. Subindex $j+m+n$ in  $\mathbf{F}_{i+\frac{1}{2}}$, $\mathbf{D}_{i+\frac{1}{2}}$, $(\mathbf{B}\partial \mathbf{Q})_{i}$ and $\mathbf{S}_{i}$ represent the $(j+m+n)th$ component. For computing the integrals (\ref{integrals:eq-1}) and (\ref{integrals:eq-2}) we require also the state variables, but those are used from the record carried out in step 1).
Index $n+1$ is consistent with the backward evolution in time, that means starting from $t^0 = T$, we do $t^{n+1} = t^n - \Delta t$.
So in a finite number of steps, $n_T$, we reach $t^{n_T} = 0$.

\item[3)] Compute the gradient of the cost function (\ref{eq:functional:0}) using variables associated to the adjoint variables as well as state variables. In this paper as shown formally in (\ref{FormalGrad:eq-1}), they normally have the form $\nabla J_i = \mathbf{P}_i^{n_T}$ (expressed in a  discrete form which is normally used in simulations). Notice that this procedure can have more general structures and may also depend on the state variables, so let us express the gradient formally in terms of a functional, $G(\mathbf{U},\mathbf{P},\mathbf{b})$. So the update of the sought parameters has the form
\begin{eqnarray}
\label{eq:updateB}
\mathbf{b}^{k+1}_i = \mathbf{b}^{k}_i - \lambda_{IP} \cdot G(\mathbf{U},\mathbf{P},\mathbf{b}^k)_i\;,
\end{eqnarray}
where $\lambda_{IP} $ is a prescribed constant value. Notice that, the procedure is carried out for finding parameters. However, for another type of problems, for example for finding initial conditions, the procedure is quite similar to the present one.

\item[ 4)] Stop the procedure using some stop criterion. This is normally carried in terms of the relative error between two successive approximations.

 \end{itemize}   

\section{Numerical examples}\label{sec:numerical-examples}
In this section, we solve PDE-constrained optimization problems using the numerical scheme described in section \ref{section:numericalmethod}. In order to assess the performance of the present scheme, we compare with a conventional solver and so we mimic the strategy usually employed for solving this type of problems, see the reference scheme in Appendix \ref{sec:reference-scheme}. 

Simulations carried out in this section are obtained with the stable time step, $\Delta t$, computed as
\begin{eqnarray}
\label{stable:timestep}
\begin{array}{c}
\Delta t = C_{cfl} \frac{ \Delta x}{ \lambda_\infty} \;,
\end{array}
\end{eqnarray}
where 
\begin{eqnarray}
\begin{array}{c}
\lambda_{\infty} = \underset{ i}{\max} \underset{j = 1,...,2(m+n)}{ \max }  | \lambda_{j} (\mathbf{H}(x_i) ) |\;,
\end{array}
\end{eqnarray}
with $ \lambda_{j}$ being the $j$th eigenvalue of the Jacobian matrix and $\mathbf{H}(x_i)$ the initial condition  evaluated at the barycentre of cell $[x_{i-\frac{1}{2}}, x_{i+\frac{1}{2}}]$, we choose $C_{cfl} = 0.1$.  Notice that, if the initial condition $\mathbf{H}(x)$ changes as a consequence of the change in the sought parameters, the time step must also change accordingly to (\ref{stable:timestep}).

\subsection{The Burgers equation}
Let us consider the optimal problem consisting of finding the initial condition $h_0(x)$ such that 
the solution of the following hyperbolic conservation law
\begin{eqnarray}
\label{eq:burg-forw:1}
\begin{array}{c}
\partial_t q + \partial_x( \frac{q^2}{2}  ) = 0\;, t\in[0,T] \;,
\\
q(x,0 ) = h_0(x) \;,
\end{array}
\end{eqnarray}
minimizes the functional 
\begin{eqnarray}
\label{eq:Jac-bur:1}
J(h_0) = \frac{1}{2}\int_{0}^{T} \int_{-\infty}^{\infty} (q(x,t)-\bar{q}(x,t))^2 dx dt \;,
\end{eqnarray}
where $\bar{q}$ is some measurement chosen to be the exact solution of the following Burgers equation at $t=T$
\begin{eqnarray}
\label{eq:burg-forw:2}
\begin{array}{c}
\partial_t q + \partial_x( \frac{q^2}{2}  ) = 0\;, t\in[0,T] \;,
\\
q(x,0 ) = \bar{h}_0(x) \;,
\end{array}
\end{eqnarray}
where $\bar{h}_0(x)$ is a prescribed function.  These problems are endowed with transmissive boundary conditions. 
The adjoint problem is given by
\begin{eqnarray}
\label{eq:burg-back:1}
\begin{array}{c}
\partial_t p + q \partial_x p = \bar{q}- q \;,t\leq T \;, \\
p(x,T) = 0 \;.
\end{array}
\end{eqnarray}
See \cite{Lecaros:2014a} for further details concerning the derivation of the adjoint problem. The gradient of the cost function is 
\begin{eqnarray}
\label{burger:jacobian-j}
\begin{array}{c}
\nabla J(h_0) = p(x,0)\;.
\end{array}
\end{eqnarray}

The unified formulation, as proposed in this paper takes the form
\begin{eqnarray}
\label{eq:system:1}
\begin{array}{c}
\partial_t \mathbf{Q} + \partial_x \mathbf{F}(\mathbf{Q} ) + \mathbf{B}(\mathbf{Q}) \partial_x \mathbf{Q} = \mathbf{S} ( \mathbf{Q} ) \;,\\
\end{array}
\end{eqnarray}
with 
\begin{eqnarray}
\begin{array}{c}

\mathbf{Q} =\left[
\begin{array}{c}
q \\
p
\end{array}
\right],
\;

\mathbf{F}( \mathbf{Q} )
 =\left[
\begin{array}{c}
\frac{q^2}{2} \\
0
\end{array}
\right],
\;

\mathbf{S}( \mathbf{Q} )
 =\left[
\begin{array}{c}
0 \\
\bar{q} - q
\end{array}
\right],
\;

\end{array}
\end{eqnarray}
and
\begin{eqnarray}
\begin{array}{c}
\mathbf{B}(\mathbf{Q})
=
\left[
\begin{array}{cc}
0 & 0 \\
0 & q 
\end{array}
\right]\;.
\end{array}
\end{eqnarray}

System (\ref{eq:system:1}) is solved from $t=0$ up to $t=T$, thus we need an initial condition at $t = 0$. Notice that $p(x,t)$ has a kind of initial condition only at $t = T$. At $t=0$ we have to set information for $p(x,t)$. However, as $p$ does not influences the evolution of $q$ we set $p(x,0) = p_0 = 0.$  So, the initial condition is $\mathbf{Q}(x,0) = [h_0(x), 0]^T$. On the other hand, at $t= T$, the second component is set to zero which means
$$\mathbf{Q}(x,T) =[ q(x,T), 0 ]^T \;.$$

In this way, the evolution of the adjoint variable begins with the right initial condition as required in (\ref{eq:burg-back:1}).  We are going to implement the procedure for continuous as well as discontinuous initial conditions.

In order to assess the performance of the solver we compute the error given by $ Error = max_{i} | h_{0,i} - \bar{h}_0(x_i)|$, where $h_{0,i}$ is the approximated initial condition at the cell $[x_{i-\frac{1}{2}}, x_{i +\frac{1}{2} }]$.

\subsubsection{The Burgers equation with discontinuous initial condition}
Let us consider the case in which the aim is to find a discontinuous initial condition. So we setup the optimization problem as follows. We set a starting guess for the procedure, in this case we set $h_0(x) = - 0.2 $. We solve both constraint PDE and the adjoint problem on the computational domain $[-3,3]$ using $200$ uniform cells. The expected initial condition is 
\begin{eqnarray}
\begin{array}{c}
\bar{h}_0 (x) = 
\left\{
\begin{array}{c}
\frac{1}{2} \;, x < 0 \;, \\ 
0           \;, x \geq 0 \;. \\ 
\end{array}
\right.
\end{array}
\end{eqnarray}
 
To update the sought initial condition we use (\ref{eq:updateB}) with $\lambda_{IP} = 0.7$ and $ G(\mathbf{U},\mathbf{P},\mathbf{b}^k)$ given by (\ref{burger:jacobian-j}).  We also implement the reference scheme described in appendix \ref{sec:reference-scheme} and using the same information described above.

The results are depicted in  Figure \ref{fig:disc-Burguers} at the $80$th iteration. The unified formulation has  $ Error= 6.05 \cdot 10^{-2} $ whereas the solution with the reference solution called here as the conventional solver has  $Error = 0.11$.  Figure    
\ref{fig:disc-Burguers-error} depicts the error at each iteration of the global procedure.  We observe that the error is not so good for both methods even when the unified formulation provides the best approximation, it shows undershoot and overshoot at the interface position but the undershoot is less than that resulting from the conventional solver.

\begin{figure}
\includegraphics[scale=0.5]{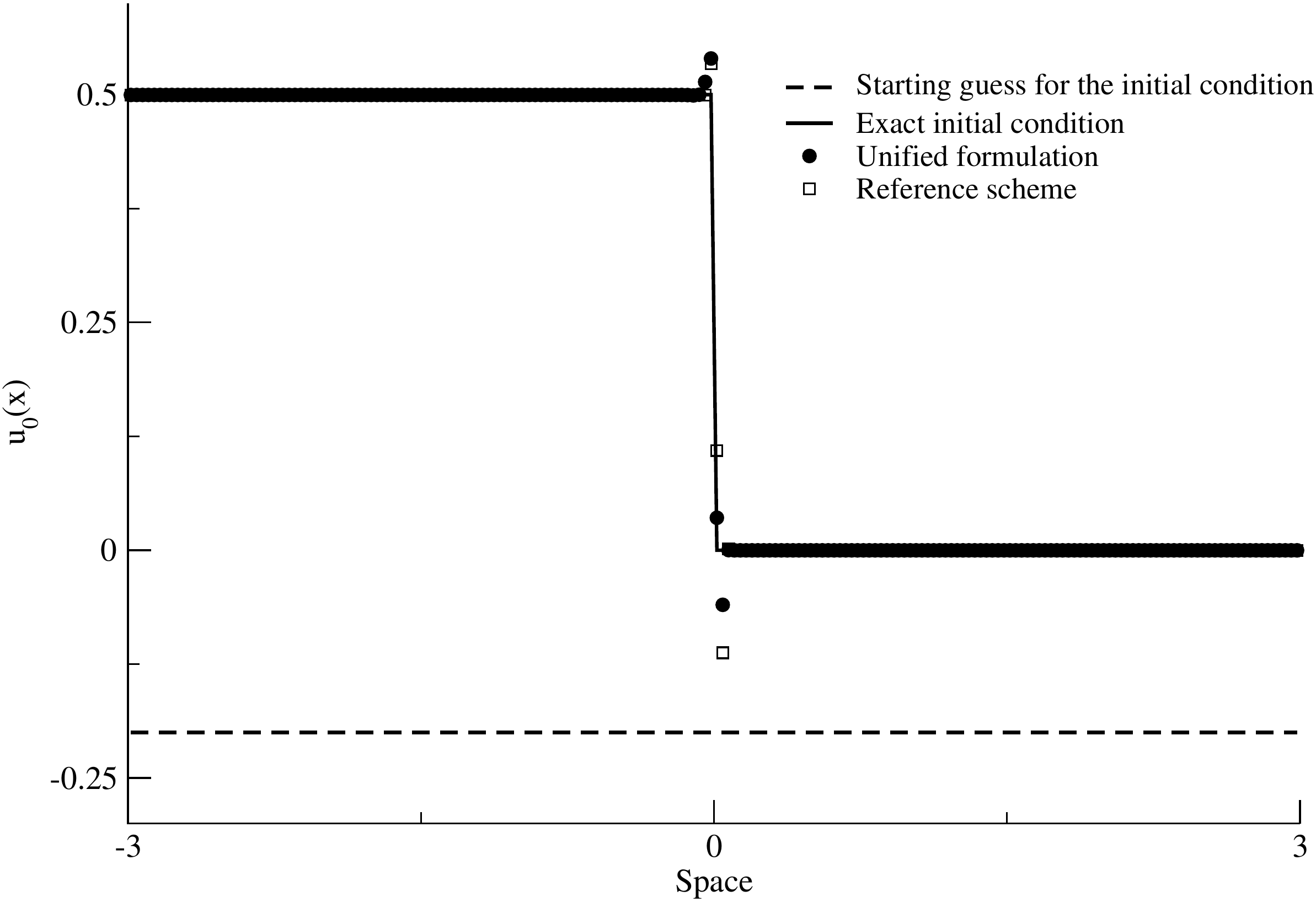}
\caption{Burgers, discontinuous initial condition. Sought initial condition, results of unified (circles) and conventional (squares) formulations. The constraint PDE problem is evolved up to $T = 0.12$.}\label{fig:disc-Burguers}
\end{figure}

\begin{figure}
\includegraphics[scale=0.5]{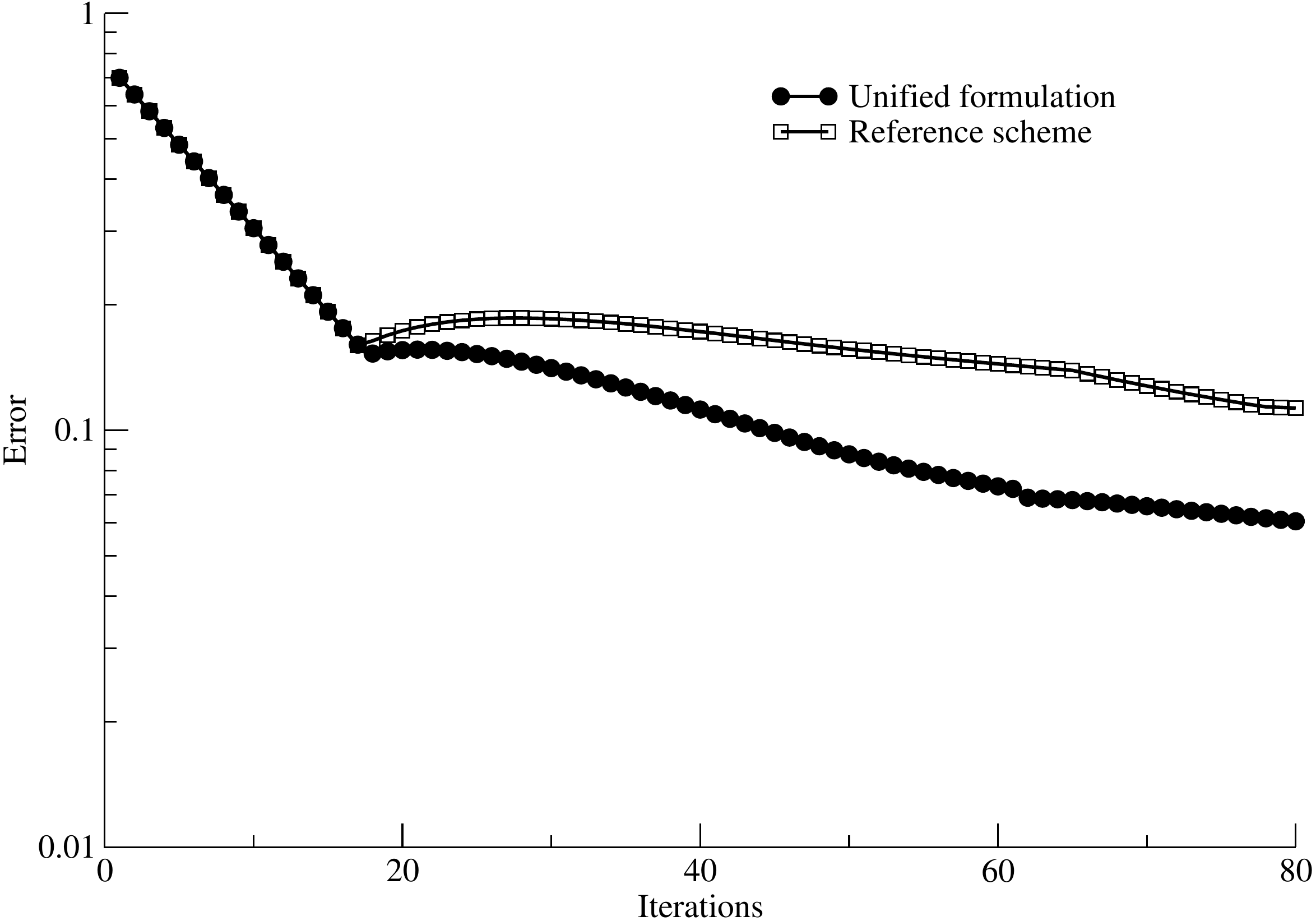}
\caption{Burgers, discontinuous initial condition. Comparison between errors associated with the unified formulation (squares) and conventional formulation (circles). The constraint PDE problem is evolved up to $T = 0.12$.}\label{fig:disc-Burguers-error}
\end{figure}

Note that in this test, the interface position is found and also the mean shape of the initial profile. However, we observe that the undershot and also the overshot appear near to the $18$th iteration as is shown in Figure \ref{fig:disc-Burguers-error}. Thus it propagates for a while and then it is stabilized. However, these undershoot and overshoot do not disappear.

\subsubsection{The Burgers equation with continuous initial condition}
Let us consider the case in which the aim is to find a continuous initial condition. So we setup the descent algorithm as follows. The optimization procedure is initialized with $h_0(x) = 0 $. We solve both constraint PDE and the adjoint problem on the computational domain $[0,1]$  using $160$ uniform cells. The expected initial condition is 
\begin{eqnarray}
\begin{array}{c}
\bar{h}_0 (x) = sin(2 \pi x) \;.
\end{array}
\end{eqnarray}

To update the sought initial condition we use (\ref{eq:updateB}) with $\lambda_{IP} = 2.7$ and $ G(\mathbf{U},\mathbf{P},\mathbf{b}^k)$ given by (\ref{burger:jacobian-j}).  We also implement the reference scheme described in appendix \ref{sec:reference-scheme}.

\begin{figure}
\includegraphics[scale=0.5]{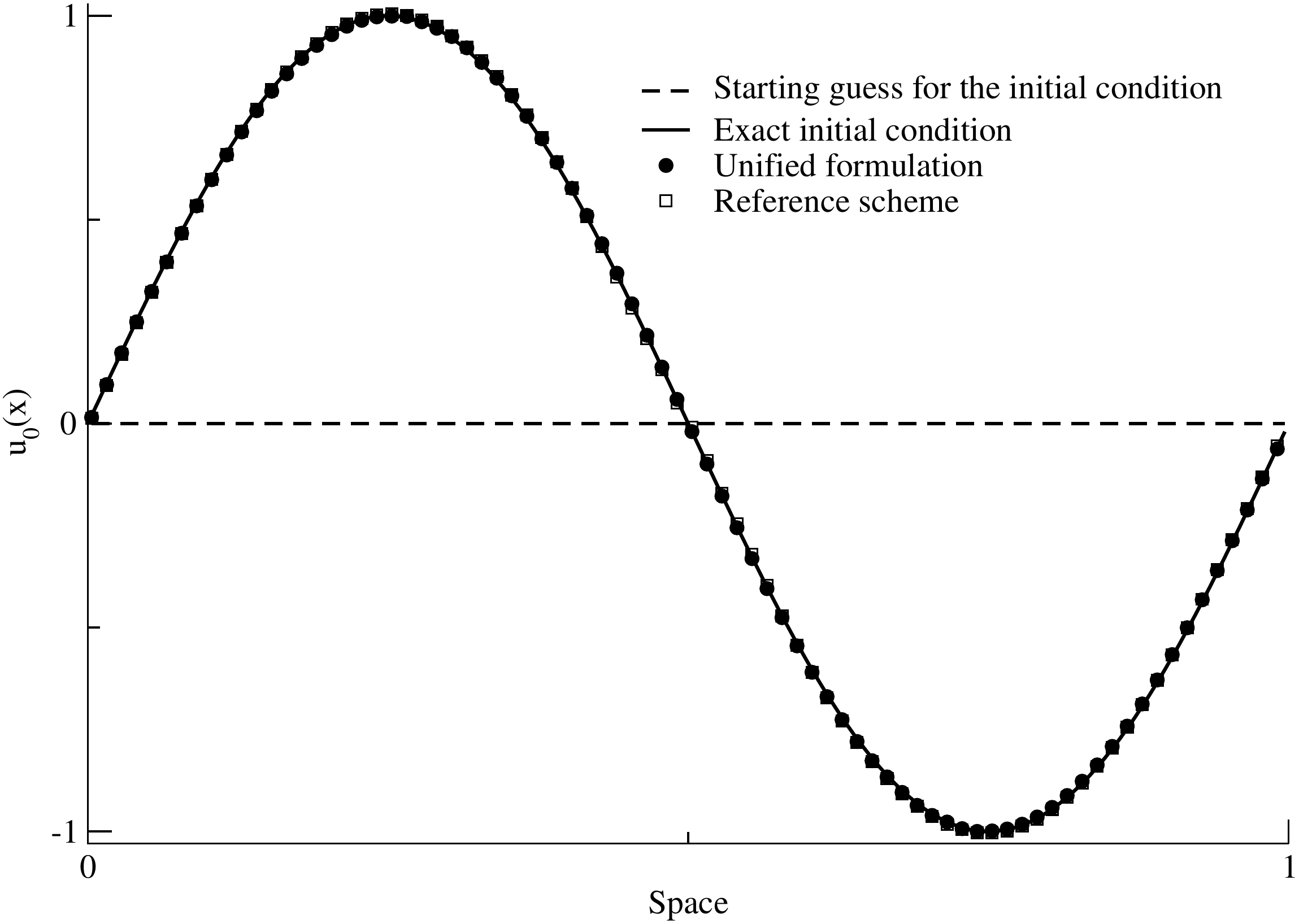}
\caption{Burgers, smooth initial condition. Sought initial condition,  results of unified (circles) and conventional (squares) formulations. The constraint PDE problem is evolved up to $T = 0.1$.} \label{fig:disc-BurguersC}
\end{figure}

\begin{figure}
\includegraphics[scale=0.5]{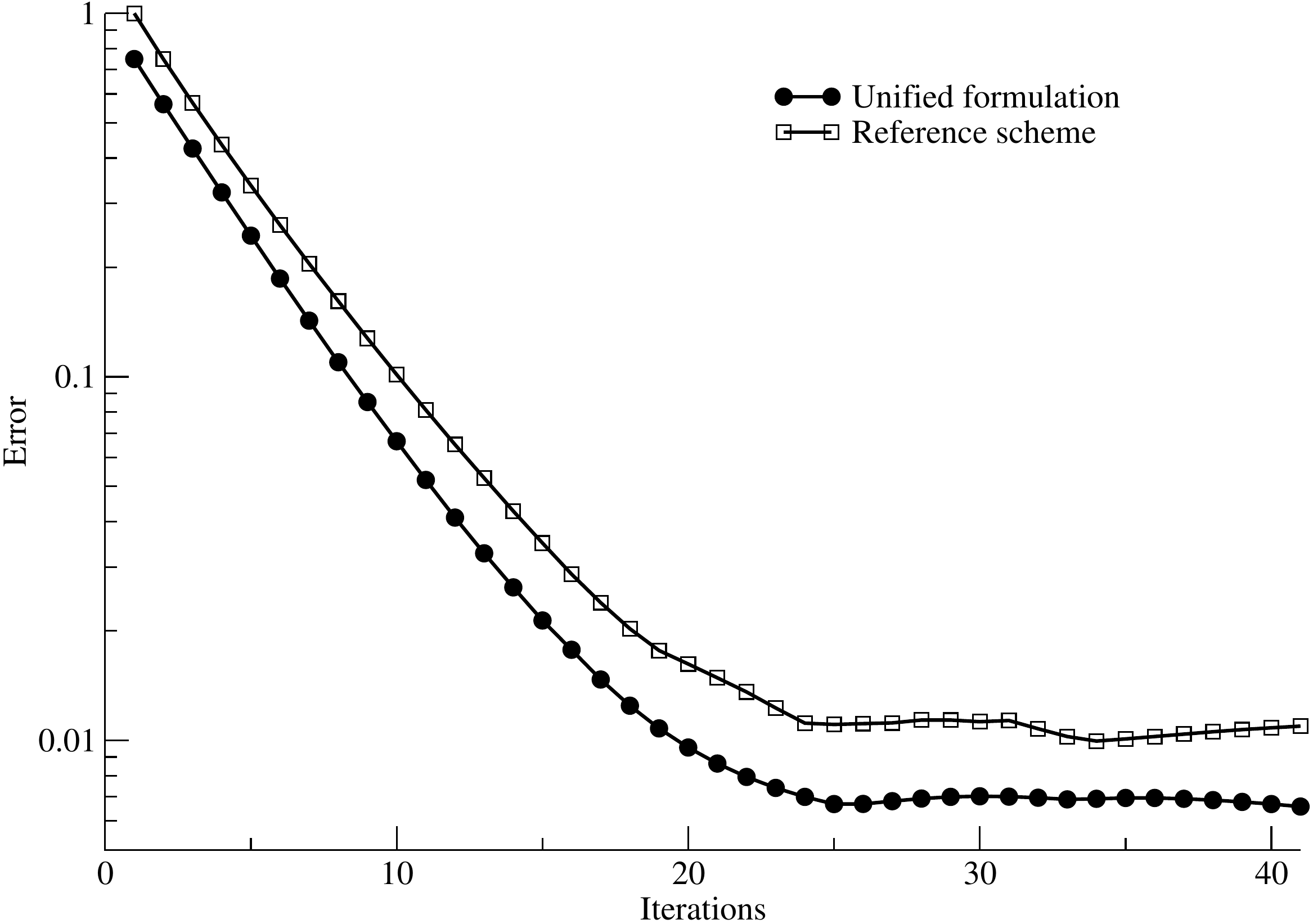}
\caption{Burgers, smooth initial condition. Comparison between errors associated with the unified formulation (squares) and conventional formulation (circles). The constraint PDE problem is evolved up to $T = 0.1$.}\label{fig:disc-BurguersC-error}
\end{figure}
The results are depicted in Figure \ref{fig:disc-BurguersC} at the $40$th iteration. The unified formulation has $ Error= 6.68 \cdot 10^{-3} $ whereas the solution with the reference scheme has $Error =1.08\cdot 10^{-2} $.  Figure \ref{fig:disc-BurguersC-error} depicts the error at each iteration of the global procedure.  We observe that the error for the present unified strategy is one-order of magnitude more accurate than the reference scheme.   If the variant described in section \ref{sect:variation} is not included, that means conventional finite volume implementation, the numerical scheme presented in section \ref{section:numericalmethod}) is of second order of accuracy. It seems to be that the high accuracy is still present even if the state variable is frozen in the evolution backward in time associated with the adjoint problem.  A CPU time measurement gives us that the present scheme is almost three times more expensive than the reference scheme. However, the solution of the constraint PDE, adjoint model and approximation update, take about three seconds, (it depends on $\Delta t$), even so the scheme is very efficient in terms of computational cost.

\subsection{The shallow water equations}\label{section:sw}
In this section we deal with the solution of the shallow water equation written as follows
\begin{eqnarray}
\label{eq:sw:1}
\begin{array}{c}
\partial_t
\left[
\begin{array}{c}
h\\
q
\end{array}
\right]
+
\partial_x 
\left[
\begin{array}{c}
\varepsilon q\\
\varepsilon \frac{q^2}{h} + \frac{1}{2}\frac{h^2}{\varepsilon}
\end{array}
\right]
=
\left[
\begin{array}{c}
 0\\
-h \partial_x b(x) 
\end{array}
\right]
\;,
\end{array}
\end{eqnarray}
where $\zeta(x,t)$ and $b(x)$ are the free surface and bottom parameterization respectively, $h(x,t)= 1 + \varepsilon (\zeta - b)$ is the total depth, $u(x,t)$ is the total depth averaged velocity, $\varepsilon$ is a dimensionless parameter
and $q=hu$ \cite{Lannes:2013a}.

There is an extensive literature concerning the general water-waves equations and shallow water approximations regimes \cite{rs1997modern,Lannes:2005a,Lannes:2013a}. Among the many applications to coastal engineering we focus on the bottom detection through measurements on the free surface $\zeta$. This problem has practical interests \cite{Nersisyan01082015} as well as theoretical challenges \cite{alazard2015control}. Because of the difficulties arising in the general water-wave system, many asymptotic models are proposed to study this phenomenon \cite{boussinesq1872theorie,israwi2011large,Lannes:2013a,Peregrine:1967a}.

We are going to consider the nonlinear shallow water equations (or Saint-Venant) \cite{de1871th} for study here the inverse problem consisting of finding the bottom surface $b(x)$ such that a prescribed free surface $ \bar{\zeta}(x,t)$ is achieved.  The adjoint problem associated to this problem is
\begin{eqnarray}
\label{SW:adjoint}
\begin{array}{c}
\partial_t
\left[
\begin{array}{c}
\tilde{h}\\
\tilde{q}
\end{array}
\right]
+
\left[
\begin{array}{cc}
     0       & \frac{h}{\varepsilon} - \varepsilon \frac{ q^2}{h^2} \\
\varepsilon  & 2 \varepsilon \frac{ q }{ h}
\end{array}
\right]
\partial_x 
\left[
\begin{array}{c}
\tilde{h}\\
\tilde{q}
\end{array}
\right]
=
\left[
\begin{array}{c}
 \frac{\bar{\zeta} - \zeta}{\varepsilon} + \tilde{q} \partial_x b \\
0
\end{array}
\right]
\;,
\end{array}
\end{eqnarray}
and the gradient of this functional is given by
\begin{eqnarray}
\label{SW:gradient}
\begin{array}{c}
\nabla J = \frac{  h -1}{ \varepsilon} + b - \bar{\zeta} + (h\tilde{q} )\;.
\end{array}
\end{eqnarray}

We neglect the term $(h\tilde{q} )$ and take $\lambda_{IP}\neq 1$. In this way the scheme is forced to do more than one iteration and the computed free surface has to be very close to the expected one. 

Notice that, the bottom surface is a prescribed function, with $\partial_t \mathbf{b} (x) = 0 $ so it can be included into the model as a non-evolutionary variable, it would allow us to include non-smooth bottom surfaces. So, the model now takes the form
\begin{eqnarray}
\label{eq:sw:1-with-b}
\begin{array}{c}
\partial_t
\left[
\begin{array}{c}
h\\
q \\
b
\end{array}
\right]
+
\partial_x 
\left[
\begin{array}{c}
\varepsilon q\\
\varepsilon \frac{q^2}{h} + \frac{1}{2}\frac{h^2}{\varepsilon} \\
0
\end{array}
\right]
+
\left[
\begin{array}{ccc}
 0 & 0 & 0\\
 0 & 0 & h   \\
 0 & 0 & 0\\
\end{array}
\right]
\partial_x
\left[
\begin{array}{c}
 h\\
 q  \\
 b
\end{array}
\right]
=
\left[
\begin{array}{c}
0\\
0\\
0
\end{array}
\right]
\;,
\end{array}
\end{eqnarray}

The unified formulation for this test is given by
\begin{eqnarray}
\label{eq:system:sw-1}
\begin{array}{c}
\partial_t \mathbf{Q} + \partial_x \mathbf{F}(\mathbf{Q} ) + \mathbf{B}(\mathbf{Q}) \partial_x \mathbf{Q} = \mathbf{S} ( \mathbf{Q} ) \;,\\
\end{array}
\end{eqnarray}
where
\begin{eqnarray}
\begin{array}{c}
\mathbf{Q}
\left[
=\begin{array}{c}
h \\
q \\
b \\
\tilde{h} \\
\tilde{q}
\end{array}
\right]\;,

\mathbf{F}( \mathbf{Q} )
\left[
=\begin{array}{c}
\varepsilon q\\
\varepsilon \frac{q^2}{h} + \frac{1}{2}\frac{h^2}{\varepsilon}\\
0 \\
0 \\
0
\end{array}
\right]\;,
\\
\mathbf{B}(\mathbf{Q} ) =
\left[
\begin{array}{ccccc}
0 & 0 & 0 & 0 & 0 \\
0 & 0 & h & 0 & 0 \\
0 & 0 & 0 & 0 & 0 \\
0 & 0 & 0 & 0 &  \frac{h}{\varepsilon} - \varepsilon \frac{q^2}{h^2}  \\
0 & 0 & -\tilde{q} & \varepsilon & 2\varepsilon \frac{q}{h}
\end{array}
\right]\;,
\mathbf{S}( \mathbf{Q} )
\left[
=\begin{array}{c}
0 \\
0 \\
0 \\

\frac{\bar{\zeta}-\zeta}{\varepsilon}
\end{array}
\right]\;.
\end{array}
\end{eqnarray}

Notice that the Jacobian matrix of $\mathbf{F}$ with respect to $\mathbf{Q}$, $\mathbf{A}(\mathbf{Q})$ is given by
\begin{eqnarray}
\begin{array}{c}
\mathbf{A}(\mathbf{Q})
=
\left[
\begin{array}{ccccc}
0 & \varepsilon & 0 & 0 & 0 \\

\frac{ h }{\varepsilon} -  \frac{q^2 }{  h^2} \varepsilon & 
2\varepsilon \frac{q}{h} & 0 & 0 & 0 \\
0 & 0 & 0 & 0 & 0 \\ 
0 & 0 & 0 & 0 & 0 \\ 
0 & 0 & 0 & 0 & 0 \\ 
\end{array}
\right]\;.
\end{array}
\end{eqnarray}

Moreover, this has the decomposition $\mathbf{A} = \mathbf{R} \mathbf{\Lambda} \mathbf{R}^{-1}\;,$ with 
\begin{eqnarray}
\begin{array}{c}
\mathbf{R}(\mathbf{Q})
=
\left[
\begin{array}{ccccc}
1 & 1 & 0 & 0 & 0 \\

\frac{ q^2 \varepsilon^2 -  h^3 }{  h  q \varepsilon^2 + h^{\frac{5}{2}}  \varepsilon} & 
\frac{ q^2 \varepsilon^2 - h^3}{ h q  \varepsilon^2 - h^{\frac{5}{2}} \varepsilon} & 0 & 0 & 0 \\
0 & 0 & 1 & 0 & 0 \\ 
0 & 0 & 0 & 1 & 0 \\ 
0 & 0 & 0 & 0 & 1 \\ 
\end{array}
\right] \;,
\end{array}
\end{eqnarray}
\begin{eqnarray}
\begin{array}{c}
\mathbf{\Lambda}(\mathbf{Q})
=
\left[
\begin{array}{ccccc}
u \varepsilon - \sqrt{h} & 0 & 0 & 0 & 0 \\

0 & u \varepsilon + \sqrt{h} & 0 & 0 & 0 \\
0 & 0 & 0 & 0 & 0 \\ 
0 & 0 & 0 & 0 & 0 \\ 
0 & 0 & 0 & 0 & 0 \\ 
\end{array}
\right] \;,
\end{array}
\end{eqnarray}
\begin{eqnarray}
\begin{array}{c}
\mathbf{R}(\mathbf{Q})^{-1}
=
\left[
\begin{array}{ccccc}

\frac{ q^{2} \epsilon^{2} - { h^{3}}}{2 h^{\frac{3}{2}} q  \varepsilon - 2 h^{3} }
& 
- \frac{\varepsilon}{2 \sqrt{ h } } 
 & 0 & 0 & 0 \\

- \frac{ q^{2} \epsilon^{2} - { h^{3}}}{2 h^{\frac{3}{2}} q  \varepsilon + 2 h^{3} }

& \frac{\varepsilon}{2 \sqrt{ h } } & 0 & 0 & 0 \\
0 & 0 & 1 & 0 & 0 \\ 
0 & 0 & 0 & 1 & 0 \\ 
0 & 0 & 0 & 0 & 1 \\ 
\end{array}
\right] \;.
\end{array}
\end{eqnarray}

In this test, the aim is to find the bottom surface $b(x)$ such that the model system (\ref{eq:sw:1}) provides a free surface $\zeta(x,t)$ close to the profile $\bar{\zeta}(x,t) = \frac{   0.3  (x - t) }{ cosh(x - t)^2}$.  Here (\ref{eq:sw:1})  is endowed with transmissive boundary conditions and initial condition  $q(x,t) = 0$, $h(x,0)=1$ and $b(x) = 0.2$, see Figure \ref{fig:sw:StartingGuess}. We use the propose methodology for finding the set of condition to find the sought profile up to the final time $t_{out} = 3$. For simulations we have used $300$ cells, $C_{cfl} = 0.1$, $\varepsilon = 0.01$. To update the sought bottom surface we set $\lambda_{IP}= 1.9$. Figure  \ref{fig:sw:Converged} shows the result of the simulation, as you can see the system achieve the sough profile with a good resolution despite of the few iterations as well.  Notice that $\bar{\zeta}$ does not correspond to any exact solution of model (\ref{eq:sw:1}) with initial conditions used in this model. So this is not a test using synthetic data. However, the procedure is able to provide an acceptable approximation of the sought profile. Model system consisting of (\ref{eq:sw:1}), (\ref{SW:adjoint}) and (\ref{SW:gradient}), is solved using the scheme of reference, appendix \ref{sec:reference-scheme}. Notice also that this strategy is independent of the elected functional $J$ and can be employed to deal with more general shallow water regimes \cite{israwi2011large,Lannes:2013a,Peregrine:1967a}, including the possibility of detecting moving bottoms $b(x,t)$.

\begin{figure}
\begin{center}
\includegraphics[scale=0.5]{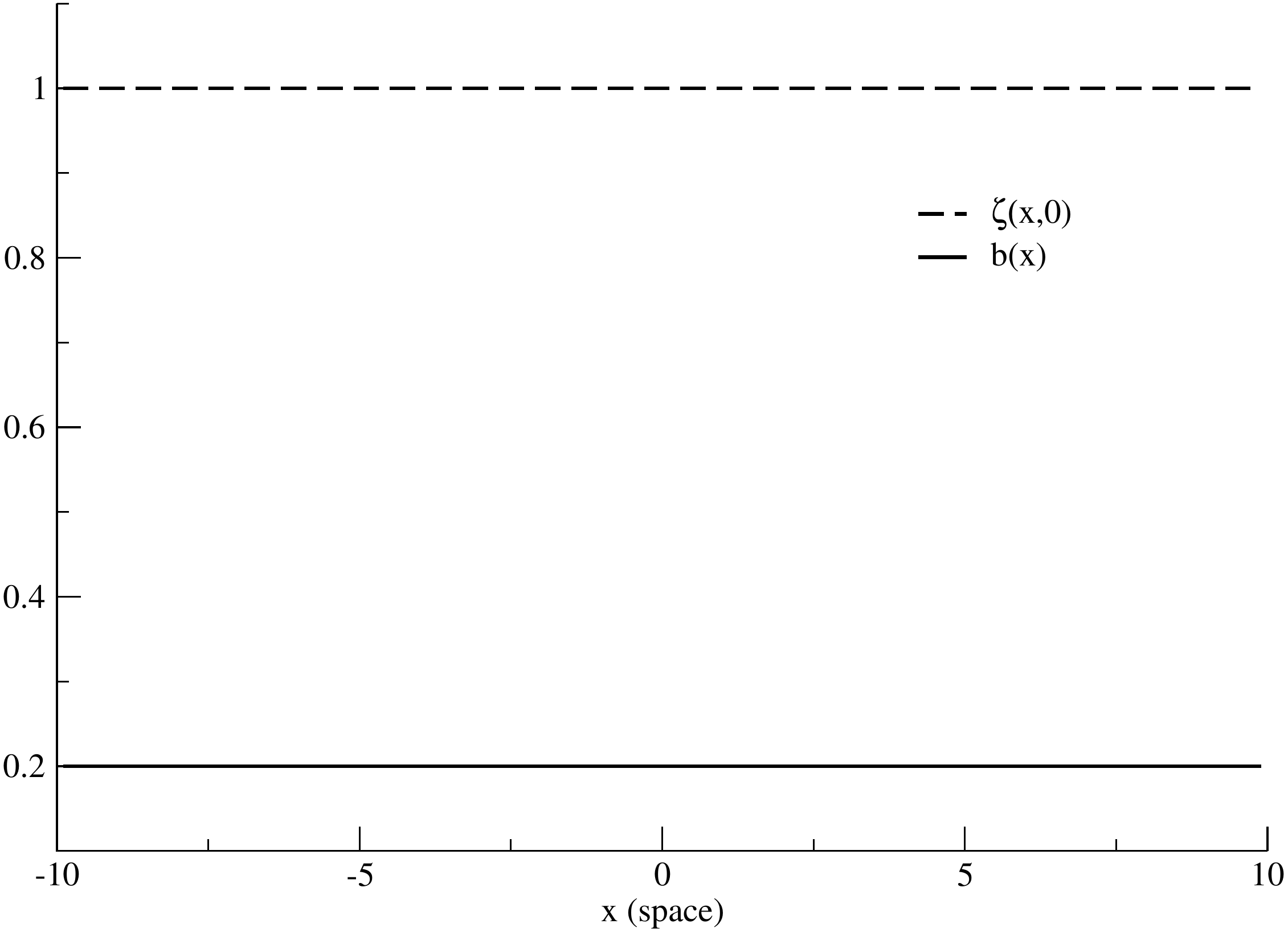}
\end{center}
\caption{Profile $\zeta(x,t)$ of the starting initial condition to initialize the global optimization procedure.}\label{fig:sw:StartingGuess}
\end{figure}
\begin{figure}
\begin{center}
\includegraphics[scale=0.5]{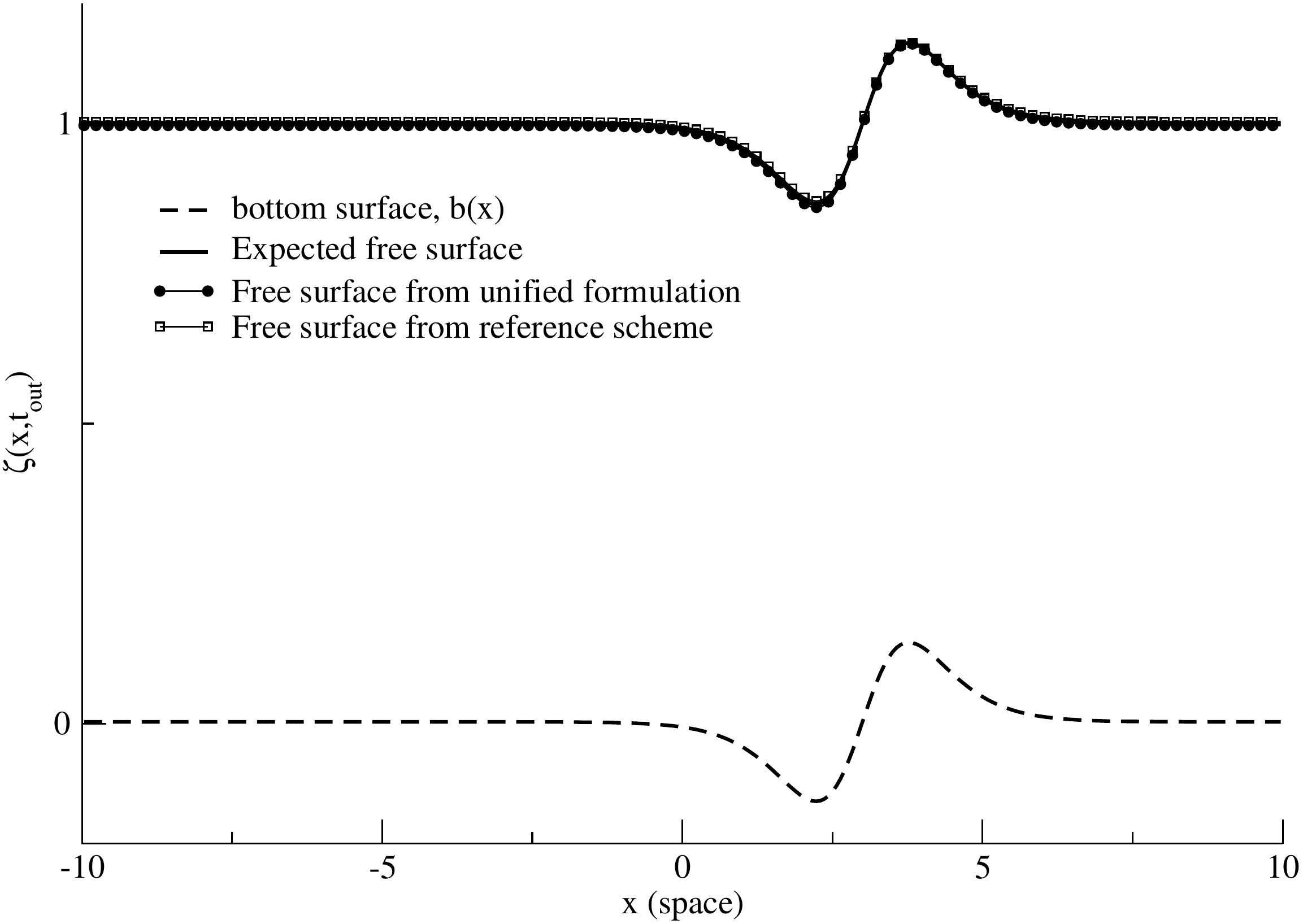}
\end{center}
\caption{Shallow water. Sought profile  for the shallow water, $\varepsilon = 0.01$. Using $300$ cells, $\lambda_{IP}= 1.9$. 
Bottom profile (thick line), free surface from unified formulation (circle), free surface from  reference scheme (square) and expected free surface (continuous line). 
At $40th$ iteration, unified formulation generates $Error =4.93 \cdot 10^{-3} $ and conventional formulation generate $Error =5.49 \cdot 10^{-3} $.   }\label{fig:sw:Converged}
\end{figure}
\begin{figure}
\begin{center}
\includegraphics[scale=0.5]{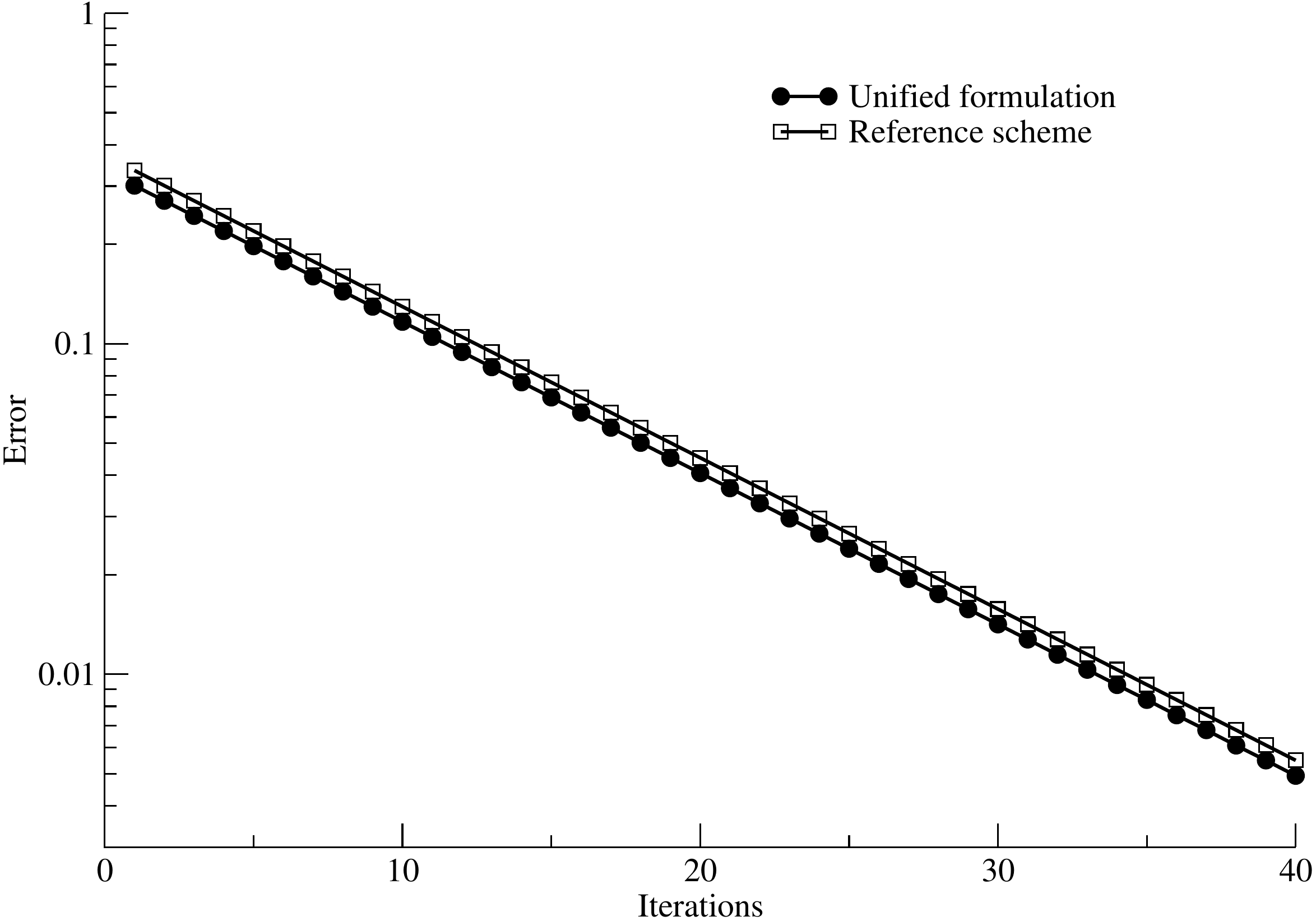}
\end{center}
\caption{Shallow water. Error measured as the distance with respect to the prescribed free surface for both, the unified formulation (square) and conventional formulation (circles).}\label{fig:sw:Error}
\end{figure}

\section{Conclusions}\label{sec:conclusions}
In the present paper, we have expressed the adjoint method for PDE-constrained optimization problems in a unified frame. A numerical scheme based on a class of high order finite volume schemes has been adapted for solving the unified system. The scheme is able to deal with non-conservative equations and we have seen that the degree of approximation is at least those of conventional solvers.  We have solved two test problems, the first one corresponds to the Burgers equations and the second one, the Nonlinear Shallow Water equations. In the case of Burgers equations, we assess two cases. First, we have solved the problem of finding an initial condition which is a discontinuous function, which generates a shock wave. The method finds the right location and also the main shape of the initial profile. We observed stationary overshot and undershot at the interface position, it is also observed by the reference scheme. Second, we recover a smooth initial condition, the constraint PDE is solved up to $T = 0.12$, before discontinuities start developing, so the solution remains continuous. In this case, the procedure has provided very good agreements in both the present and reference schemes. Regarding the shallow water equations, we also have obtained very good agreements. We have shown the evolution of the error between the computed sought parameters and that obtained in the simulations. We have observed that the present scheme always generated approximations with slightly better approximation than the reference scheme in the case of discontinuous solutions and for smooth cases, it has generated approximations whose accuracy is almost one order of magnitude.

\section*{Acknowledgements}
G.I. Montecinos thanks FONDECYT in the frame of the research project FONDECYT Postdoctorado 2016, number 3160743.

\appendix
\section{The numerical scheme of reference}\label{sec:reference-scheme}
In this section, we modify a strategy available in the literature for solving this type of problem, given by \cite{Lellouche:1994a}. The original scheme is implemented to solve an inverse problem to find information at boundaries. In this work we are interested in solving initial value problems, boundary conditions in our case are just of the transmissive type. Moreover, we set the strategy for finding model parameters $\mathbf{b}$, through out the adjoint method. So, we provide an initial parameter $\mathbf{b}$ and then we apply the following procedure. 

\begin{itemize}
\item[1)] Given $\mathbf{b}^k$, we solve the hyperbolic model  (\ref{eq:primal:0}). Tests in this paper are conservatives, we neglect source terms $\mathbf{L}$.  So, we use the following conservative scheme for marching in time up to the output time $t = T$, which is reached in a finite number of steps, let say $n_T$, that means starting from $t^0 = 0$, we do $t^{n+1} = t^{n} + \Delta t$,  $n_T$ times up to $t^{n_T} = T$.
\begin{eqnarray}
\begin{array}{c}
\mathbf{U}^{n+1}_i = \mathbf{U}^{n}_i - \frac{\Delta t}{\Delta x}   \left[  
\mathbf{R}_{i+\frac{1}{2}} - \mathbf{R}_{i-\frac{1}{2}} 
\right] + \Delta t \mathbf{L}(\mathbf{U}^{n}_i, \mathbf{b}^{k}_{i} ) \\ + \frac{\Delta t}{2 \Delta x} \tilde{\mathbf{B}}(\mathbf{U}^{n}_i, \mathbf{b}^{k}_{i} ) ( \mathbf{b}^{k}_{i+1} - \mathbf{b}^{k}_{i-1} ) \;,
\end{array}
\end{eqnarray}
where $ \mathbf{R}_{i+\frac{1}{2}} = \frac{1}{2}( \mathbf{R}( \mathbf{ U}_{i+1}^n,\mathbf{b}^k ) + \mathbf{R}(\mathbf{ U}_{i}^n,\mathbf{b}^k )  ) - \frac{\lambda_{i+\frac{1}{2}}  }{2} (\mathbf{ U}_{i+1}^n - \mathbf{ U}_{i}^n ) .$  This corresponds to the well known numerical flux of Rusanov, $ \lambda_{i+\frac{1}{2}} = \max ( \lambda (\mathbf{ U}_{i}^n), \lambda (\mathbf{ U}_{i+1}^n) )$ with $\lambda (\mathbf{ U}) $ being the maximum of the  eigenvalues (in magnitude) of the Jacobian matrix of $ \mathbf{R}$ with respect to $\mathbf{U}$. 

\item[ 2)] Once the previous step is completed. We start the backward evolution of the so called adjoint method. We use a finite difference approach.
We set the initial condition $\mathbf{P}^{0}_i = 0$.
\begin{eqnarray}
\begin{array}{c}

\mathbf{P}^{n}_i = \mathbf{P}^{n+1}_i + \frac{\Delta t}{2 \Delta x}\mathbf{ J}_R^T( \mathbf{U}^{n+1}_i ) \left[

\mathbf{P}^{n+1}_{i+1} - \mathbf{P}^{n+1}_{i-1}   
\right] - \Delta t \tilde{\mathbf{S}} ( \mathbf{U}^{n+1}_i ) \;.
\end{array}
\end{eqnarray}

Index $n+1$ is consistent with the backward evolution in time, that means starting from $t^0 = T$, we do $t^{n+1} = t^n - \Delta t$. So in a finite number of steps, $n_T$, we reach $t^{n_T} = 0$.

\item[ 3)] Update the parameter $\mathbf{b}^k$. It is carried out using the gradient of the cost functional (\ref{eq:functional:0}). The expression for the gradient of this functional depends on the problem at hand. In this paper they normally have the form $\nabla J_i = \mathbf{P}_i^{n_T}$ in terms of the discretization. Notice that this procedure can have more general structures and may also depend on the state variables, so let us express the gradient formally in terms of a functional, $G(\mathbf{U},\mathbf{P},\mathbf{b})$. So the update of the sought parameters has the form
\begin{eqnarray}
\mathbf{b}^{k+1}_i = \mathbf{b}^{k}_i -  \lambda_{IP} \cdot G(\mathbf{U},\mathbf{P},\mathbf{b}^k)_i\;,
\end{eqnarray}
where $\lambda_{IP} $ is a prescribed constant value.  Notice that, the procedure is carried out for finding parameters. However, for another type of problems, for example for finding initial conditions, the procedure is quite similar to the present one.
\item[ 4)] Stop the procedure using some stopped criterion. This is normally carried in terms of some relative error.
\end{itemize}

Let us point out that the scheme in \cite{Lellouche:1994a} is globally implicit, in turns, here we derive the global explicit version of it. This is because the proposed scheme in this paper is globally explicit.


\bibliographystyle{plain}
\bibliography{ref}

\begin{thebibliography}{10}

\bibitem{Abergel1990}
F.~Abergel and R.~Temam.
\newblock On some control problems in fluid mechanics.
\newblock {\em Theoretical and Computational Fluid Dynamics}, 1(6):303--325,
  Nov 1990.

\bibitem{alazard2015control}
Thomas Alazard, Pietro Baldi, and Daniel Han-Kwan.
\newblock Control of water waves.
\newblock {\em arXiv preprint arXiv:1501.06366}, 2015.

\bibitem{Andersson:1998a}
Paul Andersson and Martin Berggren.
\newblock {\em {Computational Methods for Optimal Design and Control:
  Proceedings of the AFOSR Workshop on Optimal Design and Control Arlington,
  Virginia 30 September-3 October, 1997}}.
\newblock Progress in Systems and Control Theory 24, Jeff Borggaard, John
  Burns, Eugene Cliff, Scott Schreck (eds.). Birkh\"{a}user Basel, 1 edition,
  1998.

\bibitem{boussinesq1872theorie}
Joseph Boussinesq.
\newblock Th{\'e}orie des ondes et des remous qui se propagent le long d'un
  canal rectangulaire horizontal, en communiquant au liquide contenu dans ce
  canal des vitesses sensiblement pareilles de la surface au fond.
\newblock {\em Journal de Math{\'e}matiques Pures et Appliqu{\'e}es}, pages
  55--108, 1872.

\bibitem{Castro:2008a-1}
C.~Castro, F.~Palacios, and E.~Zuazua.
\newblock {An alternating descent method for the optimal control of the
  inviscid Burgers equation in the presence of shocks}.
\newblock {\em Mathematical Models and Methods in Applied Sciences},
  18(3):369--416, 2008.

\bibitem{Castro:2008a}
C.~E. Castro and E.~F. Toro.
\newblock {Solvers for the high--order Riemann problem for hyperbolic balance
  laws}.
\newblock {\em Journal of Computational Physics}, 227:2481--2513, 2008.

\bibitem{CastroM:2006a}
Manuel~J. Castro, Jos{\'{e}}~M. Gallardo, and Carlos Par{\'{e}}s.
\newblock High order finite volume schemes based on reconstruction of states
  for solving hyperbolic systems with nonconservative products. applications to
  shallow-water systems.
\newblock {\em Math. Comput.}, 75(255):1103--1134, 2006.

\bibitem{de1871th}
A.~Barr{\'e} de~Saint-Venant.
\newblock Th\'eorie du mouvement non permanent des eaux, avec application aux
  crues des rivi\`eres et \`a lintroduction des mar\'ees dans leur lit.
\newblock {\em Comptes Rendus des s\'eances de l'Acad\'emie des Sciences},
  73:237--240, 1871.

\bibitem{Dumbser:2010b}
M.~Dumbser and E.~F. Toro.
\newblock {A simple extension of the Osher Riemann solver to non-conservative
  hyperbolic systems}.
\newblock {\em Journal of Scientific Computing}, 48:70--88, 2010.

\bibitem{Dumbser:2014a}
M.~Dumbser, O.~Zanotti, R.~Loubere, and S.~Diot.
\newblock {A posteriori subcell limiting of the discontinuous Galerkin finite
  element method for hyperbolic conservation laws}.
\newblock {\em Journal of Computational Physics}, 278:47--75, 2014.

\bibitem{Friedrichs:1971a}
K.~O. Friedrichs and P.~D. Lax.
\newblock Systems of conservation equations with a convex extension.
\newblock {\em Proceedings of the National Academy of Sciences of the United
  States of America}, 68:1686--1688, 1971.

\bibitem{Godunov:1961a}
S.~K. Goduvov.
\newblock An interesting class of quasilinear systems.
\newblock {\em Doklady Akademii nauk SSSR}, 139:521--523, 1961.

\bibitem{Gunzburger:2003a}
Max~D. Gunzburger.
\newblock {\em Perspectives in flow control and optimization}.
\newblock Advances in design and control. Society for Industrial and Applied
  Mathematics, 2003.

\bibitem{Harten:1987a}
A.~Harten and S.~Osher.
\newblock {Uniformly High-Order Accurate Nonoscillatory Schemes. I}.
\newblock {\em SIAM Journal on Numerical Analysis}, 24(2):279--309, 1987.

\bibitem{Hinze:2008a}
Michael Hinze, Rene Pinnau, Michael Ulbrich, and Stefan Ulbrich.
\newblock {\em Optimization with PDE Constraints}.
\newblock Mathematical Modelling: Theory and Applications. Springer, 1 edition,
  2008.

\bibitem{israwi2011large}
Samer Israwi.
\newblock Large time existence for 1d green-naghdi equations.
\newblock {\em Nonlinear Analysis: Theory, Methods \& Applications},
  74(1):81--93, 2011.

\bibitem{James:2008a}
F.~James and M.~Postel.
\newblock Numerical gradient methods for flux identification in a system of
  conservation laws.
\newblock {\em Journal of Engineering Mathematics}, 60(3-4):293--317, 2008.

\bibitem{James:1999a}
F.~James and M.~{Sep\'{u}lveda}.
\newblock Convergence results for the flux identification in a scalar
  conservation law.
\newblock {\em SIAM Journal on Control and Optimization}, 37(3):869--891, 1999.

\bibitem{rs1997modern}
RS~Johnson.
\newblock {\em A modern introduction to the mathematical theory of water
  waves}, volume~19.
\newblock Cambridge University Press, 1997.

\bibitem{Kang:2005a}
H.~Kang and K.~Tanuma.
\newblock Inverse problems for scalar conservation laws.
\newblock {\em Inverse Problems}, 21:1047--1059, 2005.

\bibitem{Knopoff:2013a}
D.A. Knopoff, D.R. Fernández, G.A. Torres, and C.V. Turner.
\newblock Adjoint method for a tumor growth pde-constrained optimization
  problem.
\newblock {\em Computers \& Mathematics with Applications}, 66(6):1104 -- 1119,
  2013.

\bibitem{Lannes:2005a}
D.~Lannes.
\newblock Well-posedness of the water-waves equations.
\newblock {\em Journal of the American Mathematical Society}, 18(3):605--654,
  2005.

\bibitem{Lannes:2013a}
D.~Lannes.
\newblock {\em The Water Waves Problem: Mathematical Analysis and Asymptotics}.
\newblock Mathematical Surveys and Monographs. American Mathematical Society,
  2013.

\bibitem{lax1954weak}
Peter~D Lax.
\newblock Weak solutions of nonlinear hyperbolic equations and their numerical
  computation.
\newblock {\em Communications on pure and applied mathematics}, 7(1):159--193,
  1954.

\bibitem{Lecaros:2014a}
R.~Lecaros and R.~Zuazua.
\newblock Control of 2d scalar conservation laws in the presence of shocks.
\newblock {\em Mathematics of computation. Submitted}, 2014.

\bibitem{Lellouche:1994a}
J.-M. Lellouche, J.-L. Devenon, and I.~Dekeyser.
\newblock Boundary control of burgers' equation—a numerical approach.
\newblock {\em Computers \& Mathematics with Applications}, 28(5):33 -- 44,
  1994.

\bibitem{Montecinos:2012a}
G.~Montecinos, C.~E. Castro, M.~Dumbser, and E.~F. Toro.
\newblock {Comparison of solvers for the generalized Riemann problem for
  hyperbolic systems with source terms}.
\newblock {\em Journal of Computational Physics}, 231:6472--6494, 2012.

\bibitem{moser1966rapidly}
J{\"u}rgen Moser.
\newblock A rapidly convergent iteration method and non-linear partial
  differential equations-i.
\newblock {\em Annali della Scuola Normale Superiore di Pisa-Classe di
  Scienze}, 20(2):265--315, 1966.

\bibitem{Zuazua:2015a}
Hayk Nersisyan, Denys Dutykh, and Enrique Zuazua.
\newblock Generation of {2D} water waves by moving bottom disturbances.
\newblock {\em IMA Journal of Applied Mathematics}, 80(4):1235--1253, 2015.

\bibitem{Nersisyan01082015}
Hayk Nersisyan, Denys Dutykh, and Enrique Zuazua.
\newblock Generation of 2d water waves by moving bottom disturbances.
\newblock {\em IMA Journal of Applied Mathematics}, 80(4):1235--1253, 2015.

\bibitem{Osher:1982a}
S.~Osher and F.~Solomon.
\newblock Upwind difference schemes for hyperbolic systems of conservation
  laws.
\newblock {\em Mathematics of Computation}, 38(158):339--374, apr 1982.

\bibitem{Pares:2006a}
P.~Par\'{e}s.
\newblock Numerical methods for nonconservative hyperbolic systems: a
  theoretical framework.
\newblock {\em SIAM Journal on Numerical Analysis}, 44(1):300--321, 2006.

\bibitem{Peregrine:1967a}
D.~H. Peregrine.
\newblock Long waves on a beach.
\newblock {\em Journal of Fluid Mechanics}, 27(2):815--827, 1967.

\bibitem{serre2015relative}
Denis Serre and Alexis~F Vasseur.
\newblock About the relative entropy method for hyperbolic systems of
  conservation laws.
\newblock 2015.

\bibitem{taylor2010partial}
M.~Taylor.
\newblock {\em Partial Differential Equations III: Nonlinear Equations}.
\newblock Applied Mathematical Sciences. Springer New York, 2010.

\bibitem{Toro:2009a}
E.~F. Toro.
\newblock {\em Riemann Solvers and Numerical Methods for Fluid Dynamics: A
  Practical Introduction}.
\newblock Springer-Verlag, third edition, 2009.
\newblock ISBN 978-3-540-25202-3.

\bibitem{Toro:2001c}
E.~F. Toro, R.~C. Millington, and L.~A.~M. Nejad.
\newblock {Towards very high--order Godunov schemes}.
\newblock In {\em Godunov Methods: Theory and Applications. Edited Review, E.
  F. Toro (Editor)}, pages 905--937. Kluwer Academic/Plenum Publishers, 2001.

\bibitem{Toro:2002a}
E.~F. Toro and V.~A. Titarev.
\newblock {Solution of the generalised Riemann problem for advection--reaction
  equations}.
\newblock {\em Proceedings of the Royal Society of London A}, 458:271--281,
  2002.

\bibitem{Ulbrich:2002a}
S.~Ulbrich.
\newblock A sensitivity and adjoint calculus for discontinuous solutions of
  hyperbolic conservation laws with source terms.
\newblock {\em SIAM Journal on Control and Optimization}, 41(3):740--797, March
  2002.

\end{thebibliography}

\end{document}